\documentclass{amsart}

\usepackage{amssymb}
\usepackage{amsmath}
\usepackage{amsthm}
\usepackage{amsbsy}
\usepackage{bm}
\usepackage{hyperref}
\date{\today}
\usepackage{cite}
\usepackage{array}

\usepackage{tikz-cd,tikz,pgf}
\usepackage{amssymb}
\usepackage{float}
\usepackage{graphicx} % Required for inserting images

\usepackage{amsthm}
\usepackage{natbib}
\usepackage{hyperref}

\theoremstyle{plain}
\newtheorem{thm}{Theorem}

\theoremstyle{definition}
\newtheorem{defn}{Definition}[section]

\newtheorem{assu}{Assumption}

\theoremstyle{theorem}
    \newtheorem{theorem}{Theorem}
    \newtheorem{lemma}[theorem]{Lemma}

    \newtheorem*{conjecture}{Conjecture}
\theoremstyle{definition} % For roman text in the body
    
    \newtheorem{fact}{Fact}
    
    \newtheorem*{remark}{Remark}
    \newtheorem{example}[theorem]{Example}
    \newtheorem{exercise}[theorem]{Exercise}

\def\eps{\epsilon}

\def\N{\mathbb{N}}

\def\<{\langle}
\def\>{\rangle}

\newcommand{\E}{\mbox{\bf E}}

\def\bar{\overline}

\newcommand\mnote[1]{} %off
\newcommand\be{\begin{equation*}}

\newcommand\ee{\end{equation*}}

\newcommand\ben{\begin{equation}}
\newcommand\een{\end{equation}}
\newcommand\bes{\begin{eqnarray*}}
\newcommand\ees{\end{eqnarray*}}

\newcommand\bex{\begin{exercise}}
\newcommand\eex{\end{exercise}}
\newcommand\beg{\begin{example}}
\newcommand\eeg{\end{example}}
\newcommand\benu{\begin{enumerate}}
\newcommand\eenu{\end{enumerate}}
\newcommand\beit{\begin{itemize}}
\newcommand\eeit{\end{itemize}}
\newcommand\berk{\begin{remark}}
\newcommand\eerk{\end{remark}}
\newcommand\bdefn{\begin{defintion}}
\newcommand\edefn{\end{definition}}
\newcommand\bthm{\begin{theorem}}
\newcommand\ethm{\end{theorem}}
\newcommand\bprf{\begin{proof}}
\newcommand\eprf{\end{proof}}
\newcommand\blem{\begin{lemma}}
\newcommand\elem{\end{lemma}}

\newcommand{\sm}{{\raise0.3ex\hbox{$\scriptstyle \setminus$}}}

\def\eps{\epsilon}

\def\CHI{\mathchoice%
{\raise2pt\hbox{$\chi$}}%
{\raise2pt\hbox{$\chi$}}%
{\raise1.3pt\hbox{$\scriptstyle\chi$}}%
{\raise0.8pt\hbox{$\scriptscriptstyle\chi$}}}
\def\smalloplus{\raise1pt\hbox{$\,\scriptstyle \oplus\;$}}

\title{convergence and asymptotic freeness of missing data matrices}

\author{Kartick Adhikari}
\address{Department of Mathematics, Indian Institute of Science Education and Research, Bhopal 462066}

\email{kartickmath [at] iiserb.ac.in}

\author{Dev Ahir}
\address{Department of Mathematics, Indian Institute of Science Education and Research, Bhopal 462066}

\email{ahir22 [at] iiserb.ac.in}

\date{\today}

\thanks{The research of KA was partially supported by the Inspire Faculty Fellowship: DST/INSPIRE/04/2020/000579. The research of DA was partially supported by the JRF-CSIR Fellowship, the Govt. of India.}


\begin{document}

\begin{abstract}
    We consider a random matrix of the form \(D_n \odot X_n\) (known as a variance profile matrix), where \(\odot\) denotes the Hadamard product of the two matrices, \(D_n\) is a deterministic matrix, and \(X_n\) is a random matrix. We call $D_n\odot X_n$ as a missing data matrix of $X_n$ when the entries of \(D_n\) are either \(0\) or \(1\). This framework is commonly used in various applied fields, such as biology, neuroscience, and network data analysis. 
    We study the convergence and asymptotic freeness of missing data matrices of iid, elliptic, and covariance random matrices. Specifically, it is known that independent iid, elliptic, and covariance matrices converge to freely independent circular, elliptic, and Marčenko-Pastur variables, respectively. In this article, we provide the necessary and sufficient conditions on deterministic matrices \(D_n\) for which these results hold true for independent missing data matrices of these three types of random matrices.

%In a recent work, Cook et al. derived Necessary and sufficient condition on $D_n$ to 
%hold the cilcular law under mild assumptions on the entries of random matrices. 
%We study the missing data matrices, effectively modeling the case of structured or missing data. 
%We show that under suitable assumptions, 
%$\frac{1}{\sqrt{n}}X_n \odot D_n$ converges to circular variable in $\ast$-distribution.  
%We also this result when $X_n$ is an elliptic matriced.
%prove that expected empirical spectral distribution (EESD) of
 %$\frac{1}{n}(X_{p \times n} \odot D_{p \times n})\cdot (X_{p \times n} \odot D_{p \times n})^\ast$ 
 %converges to Mar\v{c}enko-Pastur law with appropriate conditions. 
 %Our 
\end{abstract}

\maketitle

\section{Introduction}

The asymptotic behaviour of random matrices was first studied by Voiculescu  in the context of free probability theory \cite{voiculescu1991limit}. Later, it was investigated extensively by researchers in probability theory, analysis, and operator algebras \cite{dykema1993certain,hiai2000asymptotic,shlyakhtenko1997limit,tao2011universality,voiculescu1998strengthened}. More specifically, it has been established that, under appropriate assumptions, Wigner matrix converges to a semicircular variable \cite{wigner1955,wigner1958}, i.i.d. matrix converges to a circular variable \cite{girko1985,bai1997}, and more recent works have shown that elliptic matrices converge to elliptic variables \cite{gotze2015minimal,nguyen2015elliptic,adhikari2019brown}.

In recent years, considerable attention has been given in the literature of {\it variance profile random matrices}, that is, the random matrices of the form
$$
D_{n}\odot X_{n}
$$
where $\odot$ is the Hadamard product of two matrices, $D_n$ is an $n\times n$ deterministic matrix and $X_n$ is an $n\times n$ random matrix. Among the notable contributions, the work of Cook et al. \cite{cook2018non,cook2022non} provides a detailed analysis of the limiting spectral distribution of non-Hermitian random matrices with a variance profile. Foundational results were also established earlier by Hachem et al. \cite{hachem2006empirical}, focusing on deterministic equivalents for Gram matrices. More recent developments include the works of Adhikari et al. \cite{adhikari2019linear}, Cheliotis and Louvaris \cite{cheliotis2025limit}, and Bigot and Male \cite{bigot2021freeness}, which explore fluctuations, operator norms, and connections to free probability. This model finds applications in diverse fields such as ecology \cite{may1972will,allesina2015stability,allesina2015predicting} and neuroscience \cite{sompolinsky1988chaos,aljadeff2015transition,rajan2006eigenvalue}.

In this article, we consider deterministic matrices $ D_n  $ with each entries is either $0$ or $1$. In this case, we refer to $D_n\odot X_n $ as a {\it missing data matrix} of $X_n$. Clearly, a few entries of $X_n$ is missing (zero) in $D_n\odot X_n$. This framework naturally arises in various applications involving incomplete or partially observed data, including areas such as biological networks, neuroscience, and link prediction in complex networks \cite{picard2009deciphering,tabouy2020variational,latouche2011overlapping,lu2011link,zhao2017link}.

A random matrix is called an {\it iid matrix} if its entries are independent and identically distributed (iid) random variables. Suppose $X_n$ is a sequence of iid random matrices. Under appropriate assumptions on entries of $X_n$, we provide the necessary and sufficient conditions on the entries of $D_n$ such that  $D_{n}\odot X_{n}$ converges to a circular variable in $\ast$-distribution. Moreover, we show that independent missing data matrices of iid matrices are asymptotic free. See Theorem \ref{thm1}.  We also derive the necessary and sufficient conditions on  the deterministic matrices such that the missing data matrices of elliptic matrices converges to an elliptic variable in $\ast$-distribution. See Theorem \ref{thm2}. A random matrix is known as an {\it elliptic matrix} if its $(i,j)$-th and $(j,i)$-th entries are correlated.

{\sloppy
An other central result in random matrices is the Mar\v{c}enko-Pastur law, which characterizes the limiting spectral distribution of matrices of the form $ \frac{1}{n} X_{p \times n} X_{n \times p}^* $. In its classical formulation, this law applies when the entries of $ X_{p \times n}$ are i.i.d. Gaussian random variables \cite{marvcenko1967distribution}. Over time, this result has been extended to more general settings, including matrices with weaker moment assumptions and broader classes of distributions \cite{bai2010spectral,wachter1978strong,yin1986limiting}.  More recently, Adhikari and Bose \cite{adhikari2019brown} established the convergence of the expected empirical spectral distribution (EESD) in the case where $X_{p \times n}$ is an elliptic rectangular random matrix. In this article, we consider $Y_{p \times n}=D_{p \times n} \odot X_{p \times n}$, where $X_{p \times n}$ is an elliptic rectangular matrix and $D_{p \times n}$ is a deterministic matrix with entries $0$ or $1$. We provide necessary and sufficient conditions  on the entries of $D_{n\times p}$ to hold the  Mar\v{c}enko-Pastur law for $ \frac{1}{n} Y_{p \times n} Y_{n \times p}^* $. See Theorem \ref{thm3}.

The rest of the article is organized as follows. In Section \ref{sec2} we introduce the basic definitions. Our main three results are stated in Section \ref{sec3}. Section \ref{sec4} contains the proofs of Theorems \ref{thm1} and \ref{thm2}. The proof of Theorem \ref{thm3} is provided in Section \ref{sec5}.

\section{Preliminaries}\label{sec2}
 
In this section, we recall some basic definitions and notations that will be important for the results presented in the following discussion.
An algebra $\mathcal{A}$ is called a \textit{$\ast$-algebra} if there exists a mapping $x \mapsto x^*$ from $\mathcal{A} \to \mathcal{A}$ such that for all $x,y \in \mathcal{A}$ and $\alpha \in \mathbb{C}$, $\ast$ satisfies $(x+y)^* = x^* + y^*$, $(\alpha x)^* = \bar{\alpha} x^*$, $(xy)^* = y^* x^*$ and $(x^*)^* = x$. A pair $(\mathcal{A},\varphi)$ is called a $\ast$-probability space if $\mathcal{A}$ is a unital $\ast$-algebra with unity $1_{\mathcal{A}}$ and $\varphi$ is a positive ($\varphi(aa^\ast)\geq0,\text{ for all } a \in \mathcal{A}$) linear functional such that $\varphi(1_{\mathcal{A}})=1$. 

It is easy to see that the  space $(\mathcal{M}_{n}(\mathcal{A}), \mathrm{E}\operatorname{tr})$ is a $\ast$-probability space, where $\mathcal A$ denotes the algebra of all random variables with all moments finite, $\mathcal{M}_{n}(\mathcal{A})$  denotes the algebra of all $n \times n$ random matrices with entries from $ \mathcal{A}$, and 
\begin{align*}
    \mathrm{E}\operatorname{tr}(X_n)&=\frac{1}{n}\sum_{i=1}^n \mathrm{E}(x_{ii})\quad \text{for all } X_{n}=((x_{ij}))_{i,j=1}^n \in \mathcal{M}_{n}(\mathcal{A}).
\end{align*}
Suppose $(\mathcal{A}_n,\varphi_n),n\geq 1$ and $(\mathcal{A},\varphi)$ are $\ast$-probability spaces. Then we say that $x_n \in \mathcal{A}_n$ converges in $\ast$-distribution to $x\in \mathcal{A}$ if 
$$\lim_{n\to\infty}\varphi_n(\Pi(x_n,x_n^{\ast}))=\varphi(\Pi(x,x^{\ast}))\quad \text{for all polynomials }\Pi.$$
In particular we say that a sequence of matrices $X_n \in \mathcal{M}_{n}(\mathcal{A})$ converges to some $x \in \mathcal{A}$ in $\ast$-distribution if for any choice of $k\geq 1$ and $\epsilon_1,\epsilon_2,\dots,\epsilon_k \in \{1,\ast\}$
 $$\lim_{n \rightarrow \infty} \varphi_{n}(X_n^{\epsilon_1}X_n^{\epsilon_2}\cdots X_n^{\epsilon_k})=\varphi(x^{\epsilon_1}x^{\epsilon_2}\cdots x^{\epsilon_k}).$$
 We write the above convergence  by $X_n \xrightarrow{\ast}x$ as $n\to \infty$. Note that if $x_n$ is self-adjoint for each $n\geq 1$, then the definition reduces to the following:
 $$
 \lim_{n\to\infty}\varphi_n(x_n^k)=\varphi (x^k)\quad\text{for all }k\geq 1.
 $$
Suppose $(\mathcal{A},\varphi)$ is a $\ast$-probability space and $\mathcal{A}_1,\mathcal{A}_2,\dots,\mathcal{A}_n$ are unital $\ast$-subalgebras. We say that $\mathcal{A}_1,\mathcal{A}_2,\dots,\mathcal{A}_n$ are free in $(\mathcal{A},\varphi)$ if for any choice of $m\geq 2$ and $a_1,a_2,\dots,a_m\in \mathcal{A}$ such that $\varphi(a_i)=0,a_i \in \mathcal{A}_{j(i)}$ with $1\leq j(i) \leq n$ for $i=1,2,\dots,m$ and $j(1) \neq j(2) \neq \cdots \neq j(m)$ we have $\varphi(a_1 a_2 \cdots a_m)=0$. We say that variables $a_1,a_2,\dots,a_n \in \mathcal{A}$ are free if the unital $\ast$-subalgebras generated by them are free. Recall, the unital $\ast$-subalgebra generated by $a$ is the smallest subalgebra of $\mathcal{A}$ containing $\{1_\mathcal{A},a,a^{\ast}\}$.   
 
{\sloppy

Let $I$ be an index set, $a_n^{(i)}\in\mathcal{A}_n$ and $a_i\in\mathcal{A}$ for each $i\in I$. Then we say that $\{a_n^{(i)}:i\in I\}$ converges jointly to $\{a_i:i\in I \}$ in $\ast$-distribution if
$$\lim_{n \to \infty}\varphi_n(\Pi(\{a_n^{(i)}:i\in I\}))=\varphi(\Pi(\{a_i:i\in I \})).$$
In particular, we say that a sequence of matrices  $X_n^{(1)},X_n^{(2)},\dots,X_n^{(m)}\in \mathcal{M}_{n}(\mathcal{A})$ converges jointly to $x_1,x_2,\dots,x_m \in \mathcal{A}$ in $\ast$-distribution if for any choice of $k\geq 1$, $\epsilon_1,\epsilon_2,\dots,\epsilon_k \in \{1,\ast\}$ and $\tau_1,\tau_2,\dots,\tau_k\in[m]$, we have 
$$\lim_{n \to \infty}\varphi_n(X_n^{(\tau_1)\epsilon_1}X_n^{(\tau_2)\epsilon_2}\cdots X_n^{(\tau_k)\epsilon_k})=\varphi(x_{\tau_1}^{\epsilon_1}x_{\tau_2}^{\epsilon_2}\cdots x_{\tau_k}^{\epsilon_k}).$$
We denote the above convergence by $\{a_n^{(i)}:i\in I\}\xrightarrow{\ast}\{a_i:i\in I \}$. Moreover, if $\{a_i:i\in I \}$ are free then we say that $\{a_n^{(i)}:i\in I\}$ are asymptotically free. 
}

A variable $s$ in a $\ast$-probability space $(\mathcal{A}, \varphi)$ is called \textit{semi-circular} if it is self-adjoint and moment sequence is given by  
$$
    \varphi(s^h) =
\begin{cases} 
C_n, & \text{if } h = 2n, \\ 
0, & \text{if } h \text{ is odd},
\end{cases}
 $$
where $C_n$ are Catalan numbers. A variable  $ c \in \mathcal{A} $ is said to be \emph{circular} variable if it is given by
$
c = \frac{ s_1 + i s_2}{\sqrt{2}},
$ where \( s_1 \) and \( s_2 \) are two free semi-circular variables. A variable \( e \in \mathcal{A} \) is said to be \emph{elliptic variable with parameter} \( \rho \) if
$$
e = \sqrt{\frac{1 + \rho}{2}}\, s_1 + i \sqrt{\frac{1 - \rho}{2}}\, s_2,
$$
where \( s_1 \) and \( s_2 \) are free semi-circular variables.  
Observe that setting \( \rho = 1 \) yields a semi-circular variable, while \( \rho = 0 \) gives a circular variable.

The Mar\v{c}enko-Pastur law is defined for a parameter \( y > 0 \).  
If \( y > 1 \), it assigns a point mass of \( 1 - \frac{1}{y} \) at zero.  
Elsewhere, it has a density given by
\[
f_{\mathrm{MP}_y}(x) := 
\begin{cases}
\frac{1}{2\pi x y} \sqrt{(b - x)(x - a)}, & \text{if } a \leq x \leq b, \\
0, & \text{otherwise},
\end{cases}
\]
where, $a = a(y) = (1 - \sqrt{y})^2, \quad b = b(y) = (1 + \sqrt{y})^2$. We denote this probability distribution by $\mathrm{MP}_y$.

Suppose $X_n$ is an $n \times n$ matrix with eigenvalues $\lambda_1,\lambda_2,\dots,\lambda_n$. Then empirical spectral distribution of $X_n$ is defined by $\frac{1}{n}\sum_{i=1}^n \delta_{\lambda_{i}}$, where $\delta_a$ denotes the Dirac delta measure at $a$. The empirical spectral distribution function (ESD)
is given by
\[
F_{X_n}(x,y):=\frac{1}{n}\sum_{i=1}^n 1_{\{\Re(\lambda_i)\leq
	x,\Im(\lambda_i)\leq y\}}.
\]
 The expected empirical spectral distribution function (EESD) of $X_n$ is defined as 
\begin{align*}
    \mathrm{E}(F_{X_n}(x,y))&:=\frac{1}{n}\sum_{i=1}^n \mathrm{P}(\Re(\lambda_i)\leq x,\Im(\lambda_i)\leq y).
\end{align*}
The limiting spectral distribution (LSD) of sequence of  matrices $X_n$ is the weak limit of the sequence $\{F_{X_n}\}$, if it exists. 

%Fuglede-Kadision determinant of an invertible variable $a$ in a $\ast$-probability space $(\mathcal{A},\varphi)$ is defined by 
%$$\Delta(a) := \exp\left[ \frac{1}{2} \varphi(\log(aa^*)) \right].$$
%If $a$ is not invertible then it is defined by $\lim_{\epsilon \downarrow 0}\exp[ \frac{1}{2} \varphi(\log(aa^*+\epsilon)) ]$. The brown measure of $a$ is defined by 
%$$\mu_{B,a}:=\frac{1}{2\pi}\left(\frac{\partial^2}{\partial x^2}+\frac{\partial^2}{\partial y^2}\right)\operatorname{log}\Delta(a-\lambda),$$
%where $x$ and $y$ are real and imaginary parts of $\lambda$. Suppose $A_n$ is a an $n \times n$ matrix with eigenvalues $\lambda_1,\lambda_2,\dots,\lambda_n$. Then 
%$$\Delta(A_n)=\sqrt[n]{|det(A_n)|}\text{ and }\mu_{B,A_n}=\frac{1}{n}\sum_{i=1}^n\delta_{\lambda_{i}}.$$
%Therefore the Brown measure is the ESD of the matrix. Even if the ESD converges, it doesn’t always match the Brown measure of the limiting variable, though in some cases, they do turn out to be the same.

Moreover, we also use  the following notations in proofs.
\begin{align*}
	\mathcal{P}_{2}(2k)&:=\text{the set of all pair partitions of }\{1,2,\dots,2k\},
	\\
	\mathit{NC}_{2}(2k)&:=\text{the set of all non-crossing pair partitions of }\{1,2,\dots,2k\},
	\\
	\gamma &:= (1,2,3,\dots ,2k-1,2k) \in S_{2k}\text{ (Symmetric group)},
	\\
	\gamma\pi(r)&:=\gamma(\pi(r)),\text{where identify }\pi\in\mathcal{P}_{2}(2k) \text{ as a permutation in }S_{2k},
	\\
	I_{2k}&:=\{(i_1,i_2,\dots,i_{2k})\in\mathbb{N}^{2k}:1\leq i_t \leq n\text{ for }t=1,2,\dots,2k\},
		\\
	[m]&:=\{1,2,\dots,m\},
\end{align*}
we write $f(n)=O(g(n))\text{ if there exists }M>0\text{ and }N \in \mathbb{N} \text{ such that }|f(n)|\leq M g(n),\text{ for all }n\geq N $ and $f(n)=o(g(n))$ if for any $\epsilon>0$ there exists $N\epsilon\in\mathbb{N}$ such that $|f(n)|\leq \epsilon g(n)$ for all $n\geq N_\epsilon$.

\section{Main results}\label{sec3}
The main results  are stated in this section. Throughout, we work under the following basic assumptions.
\begin{assu}\label{assu 1}
    For every n, the variables \{$x_{i,i,n}\text{ : }1\leq i \leq n  $\} $\bigcup$ \{ ($x_{i,j,n},x_{j,i,n}$) : $1\leq i < j \leq n$\} form a collection of independent random variables, and satisfy $\mathrm{E}(x_{i,j,n})=0,$ $\mathrm{E}(x_{i,j,n}^{2})=1$ for all $i,j=1,2,\dots,n$ and $\mathrm{E}(x_{i,j,n}x_{j,i,n})=\rho$ for $1\leq i \neq j \leq n$ . Moreover,
$$\sup_{i,j,n}\mathrm{E}(|x_{i,j,n}|^{k})\leq\mathnormal{B}_{k}<\infty \text{ for all k }\geq1.$$
\end{assu}
For ease of writing, we write $x_{i,j}$ instead of $x_{i,j,n}$. The first result deals with the asymptotic freeness of independent missing data matrices of iid matrices.

\begin{thm}\label{thm1}
    Suppose $X_{n}$ is a sequence of random matrices whose entries satisfy Assumption \ref{assu 1} with $\rho = 0$ and $D_{n}$ is a sequence deterministic matrices with entries 0 or 1, then 
    $$n^{-1/2}D_{n}\odot X_{n}\xrightarrow{\ast}c \text{ if and ony if } \frac{1}{n^{2}}\sum_{i,j=1}^{n}d_{ij} \xrightarrow{} 1,$$
    where c is a circular variable.
    
    Moreover, suppose that $X_n^{(1)},X_n^{(2)},\dots,X_n^{(m)}$ are independent random matrices whose entries satisfy Assumption \ref{assu 1} with $\rho=0$ and $D_{n}^{(1)},D_n^{(2)},\dots,D_n^{(m)}$ are deterministic matrices with entries 0 or 1. Then 
    $$
    A_{n}^{(1)},\dots,A_{n}^{(m)}\xrightarrow{\ast}c_1,\dots,c_m \text{ if and ony if } \frac{1}{n^{2}}\sum_{i,j=1}^{n}d_{ij}^{(l)} \xrightarrow{} 1 \text{ for each } l\in [m],
    $$
    where $A_n^{(l)}=n^{-1/2}D_{n}^{(l)}\odot X_{n}^{(l)}$ for $l\in [m]$ and $c_1,\dots,c_m$ are free circular variables.
\end{thm}

\begin{remark}
	In a recent work \cite{cook2022non}, Cook et al. derived the necessary and sufficient condition on the doubly stochastic variance profile matrices to hold the circular law under mild assumptions on the entries of random matrices. See \cite[Corollary 2.9]{cook2022non}. In this case, the entries of deterministic matrix is strictly positive. In contrast, we work with a different model, where the deterministic matrix need not be doubly stochastic and entries can be zero.
	
	Another related work,	Wood \cite{wood2012universality} showed that the ESD of $n^{-1/2}X_n$, where $X_n$ is a random matrix and each entry is an i.i.d. copy of $\frac{1}{\sqrt{\rho}}\operatorname{I}_{\rho}Z$, converges in probability to the uniform distribution on the unit disk. Here, $Z$ is a standard complex random variable, and $\operatorname{I}_{\rho}$ is an independent Bernoulli random variable with $\rho=n^{-1+\alpha}$ for $0<\alpha\leq 1$, effectively modeling a randomly sparsified matrix. 
	
	In this work, we do not discuss the convergence of the ESD, which remains an open problem in our setting. Furthermore, if the ESD is shown to converge, a natural question is whether the limit coincides with the Brown measure of the limiting variable, which in this case is the uniform distribution on the unit disc. We have the following conjecture.
	
	\begin{conjecture}
		Let $X_n$ be an $n\times n$ iid matrix with mean zero variance one entries and $D_n$ be an $n\times n$ deterministic matrix with entries in $\{0, 1\}$. Then the LSD of  $n^{-1/2}D_n\odot X_n$ is the uniform probability measure on the unit disk almost surely if and only if 
		\[
		\frac{1}{n^{2}}\sum_{i,j=1}^{n}d_{ij} \xrightarrow{} 1, \mbox{ as $n\to \infty$}.
		\]
	\end{conjecture}
\noindent It is noteworthy of mentioning that the method used in \cite{tao2011universality}, \cite{bordenave} to prove the circular law can not be used directly. Because the fact that the smallest singular values of $X_n$ is bounded away from $0$ with high probability is used crucially in the proof. Whereas, in our set up, the smallest eigenvalue of $D_n\odot X_n$ can be zero with probability one.
	
	\end{remark}

The next two results deal with the asyptotic freeness of independent missing data matrices of elliptic matrices and covariance matrices.

\begin{thm}\label{thm2}
     Suppose $E_{n}$ is a sequence of elliptic matrices whose entries satisfy Assumption \ref{assu 1} and $D_{n}$ is a sequence of symmetric deterministic matrices with entries 0 or 1, then
        $$n^{-1/2}D_{n}\odot E_{n}\xrightarrow{\ast}e \text{ if and ony if } \frac{1}{n^{2}}\sum_{i,j=1}^{n}d_{ij} \xrightarrow{} 1,$$
        where e is an elliptic variable. 
      
     Moreover, suppose that $E_n^{(1)},E_n^{(2)},\dots,E_n^{(m)}$ are independent random matrices whose entries satisfy Assumption \ref{assu 1} with $\rho_1,\rho_2,\dots,\rho_m$ respectively and $D_{n}^{(1)},\dots ,D_n^{(m)}$ are symmetric deterministic matrices with entries 0 or 1. Then  
     $$
     A_{n}^{(1)},\dots,A_{n}^{(m)}\xrightarrow{\ast}e_1,\dots,e_m \text{ if and ony if } \frac{1}{n^{2}}\sum_{i,j=1}^{n}d_{ij}^{(l)} \xrightarrow{} 1 \text{ for each } l\in[m],
     $$
     where $A_n^{(l)}=n^{-1/2}E_{n}^{(l)}\odot D_{n}^{(l)}$ for $l\in[m]$ and $e_1,e_2,\dots,e_m$ are free elliptic variables.
\end{thm}

\begin{thm}\label{thm3}
Suppose $X_{p\times n }$ is an elliptic rectangular matrix whose entries satisfy Assumption \ref{assu 1} and $D_{p \times n}$ is a deterministic matrix with entries 0 or 1. If $\frac{p}{n} \xrightarrow{}y>0$ as $p \xrightarrow{} \infty$ and $\overline{X}_p=\frac{1}{n}(D_{p \times n} \odot X_{p \times n})\cdot (D_{p \times n} \odot X_{p \times n})^\ast$ then 
$$\overline{X}_p\xrightarrow{\ast}\mathrm{MP}_y \text{ if and ony if } \frac{1}{np}\sum_{\substack{1 \leq i \leq p, \\ 1 \leq j \leq n}}d_{ij} \xrightarrow{} 1 \text{ as p } \xrightarrow{} \infty.$$
Moreover, EESD of $\overline{X}_p$ converges to $\mathrm{MP}_y$. 

In addition, if $X_{p\times n }^{(1)},\dots,X_{p\times n }^{(m)}$ are elliptic independent random matrices whose entries satisfy Assumption \ref{assu 1} and $D_{p\times n }^{(1)},\dots,D_{p\times n }^{(m)}$ are deterministic matrices with entries 0 or 1. Then $\overline{X}_p^{(1)},\dots,\overline{X}_p^{(m)}$ are  asymptotically free if and only if
 $$ \frac{1}{np}\sum_{\substack{1 \leq i \leq p, \\ 1 \leq j \leq n}}d_{ij}^{(l)} \xrightarrow{} 1 \text{ as } p  \xrightarrow{} \infty \text{ for each }l\in[m],$$
where $\overline{X}_p^{(l)}=\frac{1}{n}(D_{p\times n }^{(l)}\odot X_{p\times n }^{(l)}) \cdot (D_{p\times n }^{(l)}\odot X_{p\times n }^{(l)})^{\ast}$ for $l\in [m]$.
\end{thm}

\section{Proofs of Theorem \ref{thm1} and \ref{thm2}}\label{sec4} 
Before presenting the proofs of Theorem \ref{thm1} and \ref{thm2}, we state the following fact and lemmas, which will be used in the argument. We give the proofs of these lemmas at the end of this section.

\begin{assu}\label{assu 2}
Let $D_n=(d_{ij})$ be an $n\times n$ deterministic matrix whose  entries satisfy the following condition
\[
d_{ij}\in \{0,1\} \mbox{ for $1\leq i,j \leq n$ and } \frac{1}{n^{2}}\sum_{i,j=1}^{n}d_{ij}\xrightarrow{} 1  \text{ as }n\to \infty.
\] 
For rectangular matrices, we adapt both the assumptions in a natural way.
\end{assu}
\begin{lemma}\label{lem 1}
      Let $D_{n}=(d_{ij})_{n \times n}$ be a sequence of deterministic matrices whose entries satisfy Assumption \ref{assu 2}. Define, for $\epsilon > 0$,
      $$
      A_{\epsilon,n}=\{1\leq i \leq n : \frac{1}{n}\sum_{j=1}^{n}d_{ij} \geq 1-\epsilon \} \text{ and }B_{\epsilon,n}=\{1\leq j \leq n : \frac{1}{n}\sum_{i=1}^{n}d_{ij} \geq 1-\epsilon \}.
      $$
       Then there exist $ N_{\epsilon} \in \mathbb{N} \text{ (depending on $\eps$) such that }$ 
       \[
       |A_{\epsilon,n}|\geq  n -\epsilon n \mbox{ and } |B_{\epsilon,n}|\geq n-\epsilon n, \mbox{ for all } n \geq N_{\epsilon} .
       \]
 \end{lemma}

 \begin{lemma}\label{lem 2}
     Let $D_{n}=(d_{ij})_{n \times n}$ be a sequence of deterministic matrices whose entries satisfy Assumption \ref{assu 2} and $I_{2k}$ be the set of 2k indices in which each index varies from 1 to n with a number of independent indices equal to k+1. Then for $ \epsilon >0 $ there exists  $N_\epsilon \in \mathbb{N}$ such that 
 $$\displaystyle\sum_{I_{2k}}d_{i_{1}i_{2}}d_{i_{3}i_{4}}\cdots d_{i_{2k-1}i_{2k}}\geq (n-k\epsilon n)^{k+1}, \text{ for all } n \geq N_\epsilon.$$
\end{lemma}

\begin{lemma}\label{lem new}
	Suppose $D_{n}^{(1)},\dots,D_n^{(m)}$ are deterministic matrices whose entries satisfy Assumption \ref{assu 2} and $I_{2k}$ is the index set as defined in Lemma \ref{lem 2}. Then for any $\tau_1,\dots,\tau_k \in [m]$ and $\epsilon>0$ there exists $N_0 \in \mathbb{N}$ such that
	$$\displaystyle\sum_{I_{2k}}d_{i_{1}i_{2}}^{(\tau_1)}d_{i_{3}i_{4}}^{(\tau_2)}\cdots d_{i_{2k-1}i_{2k}}^{(\tau_k)}\geq (n-(k+1)\epsilon n)^{k+1}, \text{ for all } n \geq N_0,$$
	where $D_n^{(\tau_r)}=(d_{ij}^{(\tau_r)})$ for $r=1,2,\dots,k$.   
\end{lemma}

\noindent Now we state a few facts that will be used in the proofs of the theorems. First consider the following notations:
\begin{align*}
	a'(r, s) &= \delta_{i_{r}i_{s}}\delta_{i_{r+1}i_{s+1}}
	\mbox{ and }\;\;
	b'(r, s) = \delta_{i_{r}i_{s+1}}\delta_{i_{s}i_{r+1}}.
\end{align*}

\begin{fact}(Lemma 2 of \cite{adhikari2019brown})\label{fact 2}
	For $ a'(r, s) $, $ b'(r, s) $ as defined above and $k\geq 1$, we have
	\[
	\frac{1}{n^{k+1}} \sum_{I_{2k}} \sum_{\ell=1}^{k} \sum_{1 \leq j_1 < \cdots < j_\ell \leq k} a'(r_{j_1}, s_{j_1}) \cdots a'(r_{j_\ell}, s_{j_\ell}) b'(r_{j_{\ell+1}}, s_{j_{\ell+1}}) \cdots b'(r_{j_k}, s_{j_k}) 
	\to 0, \]
	as $n\to\infty$.
\end{fact}

\begin{fact}(Lemma 1 and Theorem 2 of \cite{adhikari2019brown})\label{fact 4}
	Let \( e \) be an elliptic variable with parameter \( \rho \) in a  $\ast$-probability space \( (\mathcal{A}, \varphi) \). Then, for \( \epsilon_1, \ldots, \epsilon_p \in \{1, *\} \), we have:
	\[
	\varphi(e^{\epsilon_1} e^{\epsilon_2} \cdots e^{\epsilon_p}) =
	\begin{cases}
		\displaystyle\sum_{\pi \in \mathit{NC}_{2}(2k)}\prod_{(r,s) \in \pi}[(1-\delta_{\epsilon_r \epsilon_s})+\delta_{\epsilon_r \epsilon_s} \rho] & \text{if } p = 2k , \\
		0 & \text{if } p = 2k + 1.
	\end{cases}
	\]
	%where, $\mathit{T}(\pi)=\#\{(r,s)\in\pi\text{ : }\delta_{\epsilon_{r}\epsilon_{s}}=1\}$.
	%\\
	Moreover, if $e_1,\dots,e_m$ are free elliptic variables with parameters $\rho_1,\dots,\rho_m$ respectively. Then for any $k\geq 1$, $\tau_1,\dots,\tau_{2k} \in [m]$, $\epsilon_1,\epsilon_2,\dots,\epsilon_{2k} \in \{1,\ast\}$ we have
	$$\varphi(e_{\tau_1}^{\epsilon_1}\cdots e_{\tau_{2k}}^{\epsilon_{2k}})=\sum_{\pi \in NC_2(2k)}\prod_{(r,s) \in \pi} [(1-\delta_{\epsilon_{r} \epsilon_s})+\delta_{\epsilon_r \epsilon_s}\rho_{\tau_r}]\delta_{\tau_r \tau_s}.$$ 
	In particular, for $\rho=1$ and $\rho=0$ give the moments of a semi-circular variable and circular  variable respectively.  
\end{fact}

\begin{fact}(Page 367 of \cite{nica2006lectures})\label{fact 1}
	Let $\pi \in \mathcal{P}_2(2k)$ and $|\gamma_\pi|$ denote the cardinality of the set of partition blocks in $\gamma_\pi$. Then $|\gamma_\pi| \leq k + 1$, and equality holds if and only if $\pi \in \mathit{NC}_2(2k)$.
\end{fact}

\noindent 
For ease of writing, we use the following notation in the proof.
\begin{align*}
{\mathit{A}_{n}} &= n^{-1/2}Y_{n}, \mbox{ where }\;\;	Y_{n}=D_{n}\odot X_{n}=(y_{ i,j})_{n \times n},
	\\
	y_{ij}^{\epsilon} &=
	\begin{cases}
		y_{i,j}       & \quad \text{if } \epsilon \text{ = 1,}\\
		y_{j,i}  & \quad \text{if } \epsilon \text{ = }
		\ast.  \end{cases}
\end{align*}

\noindent Assuming all the lemmas and facts, we proceed to prove Theorem \ref{thm1}. 

\begin{proof} [Proof of Theorem \ref{thm1}]
	{\it The condition is necessary:} Suppose $A_{n}$ converges in $\ast$-distribution to a circular variable $c$.  We have
	\[
	\varphi_{n}(A_{n}A_{n}^{\ast}) = \frac{1}{n^{2}}\mathrm{E}[\mathrm{Tr}(Y_{n}Y_{n}^{\ast})]=
	 \frac{1}{n^2}\displaystyle\sum_{i,j=1}^{n}\mathrm{E}[d_{ij}^2 X_{ij}^2].
	\]
	By Assumption \ref{assu 1} we have $\E[X_{ij}^2]=1$ for all $1\le i,j\le n$. Since $d_{ij}^2=d_{ij}$, we get
\begin{align*}
    \varphi_{n}(A_{n}A_{n}^{\ast}) 
      &= \frac{1}{n^2}\displaystyle\sum_{i,j=1}^{n}d_{ij}.
\end{align*}
On the other hand, $A_n\xrightarrow{\ast} c$  implies that $\varphi_{n}(A_{n}A_{n}^{\ast}) \to \varphi(cc^{\ast})$ as $n\to \infty$. Therefore, as  $\varphi(cc^{\ast})=1$, we have 
\[
  \frac{1}{n^2}\displaystyle\sum_{i,j=1}^{n}d_{ij} \to 1, \mbox{ as $n\to \infty$.}
\]
 Hence the condition is necessary.

\vspace{.1cm}
\noindent {\it The condition is sufficient:} Let $t\in \N$ and  $\epsilon_{1},\epsilon_{2}\dots,\epsilon_{t}\in\{1,\ast \}$. Then we have 
    \begin{align*}
        \varphi_{n} 
 ({\mathit{A}}_{n}^{\epsilon_{1}} {\mathit{A}}_{n}^{\epsilon_{2}} \cdots 
  {\mathit{A}}_{n}^{\epsilon_{t}})&= \frac{1}{n^{\frac{t}{2}+1}}\mathrm{E}[\mathrm{Tr}({\mathit{A}}_{n}^{\epsilon_{1}}{\mathit{A}}_{n}^{\epsilon_{2}}\cdots{\mathit{A}}_{n}^{\epsilon_{t}})]
    \\ &= \frac{1}{n^{\frac{t}{2}}+1}\displaystyle\sum_{\mathit{I}_{t}}\mathrm{E}[y_{i_{1}i_{2}}^{\epsilon_{1}}y_{i_{2}i_{3}}^{\epsilon_{2}}\cdots y_{i_{t}i_{1}}^{\epsilon_{t}}], 
    \end{align*}
where $\mathit{I}_{t}=\{(i_{1},i_{2},\dots,i_{t}) \in \mathbb{N}^t \text{ : }  1 \leq i_{j} \leq n, 1 \leq j \leq t   \}$.

\vspace{.2cm}
\noindent \textbf{Case (i) : }When t is odd. Then we have 
\begin{align*}
        \left| \text{ } \frac{1}{n^{\frac{t}{2}}+1}\displaystyle\sum_{\mathit{I}_{t}}\mathrm{E}[y_{i_{1}i_{2}}^{\epsilon_{1}}y_{i_{2}i_{3}}^{\epsilon_{2}}\cdots y_{i_{t}i_{1}}^{\epsilon_{t}}] \text{ }\right| &\leq   \frac{1}{n^{\frac{t}{2}}+1}\displaystyle\sum_{\mathit{I}_{t}}\mathrm{E}[|x_{i_{1}i_{2}}^{\epsilon_{1}}x_{i_{2}i_{3}}^{\epsilon_{2}}\cdots x_{i_{t}i_{1}}^{\epsilon_{t}}|] 
        \\
        &\leq \frac{O(n^{\lfloor\frac{t}{2}\rfloor+1})}{n^{\frac{t}{2}+1}}
       \rightarrow  0, 
\end{align*}
as $n\to \infty$. For more details see the proof of Theorem 1 in \cite{adhikari2019brown}.

\vspace{.2cm}
 \noindent \textbf{Case (ii) : } Suppose t is even say 2k. Then by Assumption \ref{assu 1} we have
 \\
 $\varphi_{n} 
 ({\mathit{A}}_{n}^{\epsilon_{1}} {\mathit{A}}_{n}^{\epsilon_{2}} \cdots 
 {\mathit{A}}_{n}^{\epsilon_{2k}})$
        \begin{align*}
  &= \frac{1}{n^{k+1}}\displaystyle\sum_{\mathit{I}_{2k}}\mathrm{E}[y_{i_{1}i_{2}}^{\epsilon_{1}}y_{i_{2}i_{3}}^{\epsilon_{2}}\cdots y_{i_{2k}i_{1}}^{\epsilon_{2k}}] 
  \\
  &= \frac{1}{n^{k+1}}\displaystyle\sum_{\mathit{\pi}\in\mathcal{P}_{2}(2k)}\displaystyle\sum_{\mathit{I}_{2k}(\pi)}\displaystyle\prod_{(r,s)\in\pi}\mathrm{E}[y_{i_{r}i_{r+1}}^{\epsilon_{r}}y_{i_{s}i_{s+1}}^{\epsilon_{s}}] + o(1)
  \\
  &= \frac{1}{n^{k+1}}\displaystyle\sum_{\mathit{\pi}\in\mathcal{P}_{2}(2k)}\displaystyle\sum_{\mathit{I}_{2k}(\pi)}\displaystyle\prod_{(r,s)\in\pi}d_{i_{r}i_{r+1}}^{\epsilon_{r}}d_{i_{s}i_{s+1}}^{\epsilon_{s}}(\delta_{\epsilon_{r}\epsilon_{s}}a'(r,s) +(1-\delta_{\epsilon_{r}\epsilon_{s}})b'(r,s)) + o(1).
\end{align*}
Note that $d_{i_{r}i_{r+1}}^{\epsilon_{r}}d_{i_{s}i_{s+1}}^{\epsilon_{s}}\delta_{\epsilon_{r}\epsilon_{s}} \leq 1$ and using the Fact \ref{fact 2} we have
\\
$\lim_{n\rightarrow\infty}\varphi_{n} 
({\mathit{A}}_{n}^{\epsilon_{1}} {\mathit{A}}_{n}^{\epsilon_{2}} \cdots {\mathit{A}}_{n}^{\epsilon_{2k}})$
\begin{align*}
    &=\displaystyle\sum_{\mathit{\pi}\in\mathcal{P}_{2}(2k)}\lim_{n\rightarrow\infty}\frac{1}{n^{k+1}}\displaystyle\sum_{\mathit{I}_{2k}(\pi)}\displaystyle\prod_{(r,s)\in\pi}d_{i_{r}i_{r+1}}^{\epsilon_{r}}d_{i_{s}i_{s+1}}^{\epsilon_{s}}(1-\delta_{\epsilon_{r}\epsilon_{s}})b'(r,s)
 \\
 &=\displaystyle\sum_{\mathit{\pi}\in\mathcal{P}_{2}(2k)}\prod_{(r,s)\in\pi}(1-\delta_{\epsilon_{r}\epsilon_{s}})\lim_{n\rightarrow\infty}\frac{1}{n^{k+1}}\displaystyle\sum_{\mathit{I}_{2k}(\pi)}\displaystyle\prod_{(r,s)\in\pi}d_{i_{r}i_{r+1}}^{\epsilon_{r}}d_{i_{s}i_{s+1}}^{\epsilon_{s}}\delta_{i_r i_{s+1}} \delta_{i_s i_{r+1}}
 \\
 &=\displaystyle\sum_{\mathit{\pi}\in\mathcal{P}_{2}(2k)}\prod_{(r,s)\in\pi}(1-\delta_{\epsilon_{r}\epsilon_{s}})\lim_{n\rightarrow\infty}\frac{1}{n^{k+1}}\displaystyle\sum_{\mathit{I}_{2k}(\pi)}\displaystyle\prod_{r=1}^{2k}d_{i_{r}i_{r+1}}^{\epsilon_{r}}\delta_{i_r i_{\gamma \pi (r)}}.
\end{align*}
Observe that $d_{i_{r}i_{r+1}}^{\epsilon_{r}}\delta_{i_r i_{\gamma \pi (r)}} \leq 1$ and from Fact \ref{fact 1} number of independent indices in $I_{2k}(\pi)$ is at most k+1 and this upper bound is achieved if and only if $\pi$ is non-crossing. Using this we get  
\\~\\
$\lim_{n\rightarrow\infty}\varphi_{n} 
({\mathit{A}}_{n}^{\epsilon_{1}} {\mathit{A}}_{n}^{\epsilon_{2}} \cdots {\mathit{A}}_{n}^{\epsilon_{2k}})$
\begin{align*}
&=\displaystyle\sum_{{\pi}\in\mathit{NC}_{2}(2k)}\prod_{(r,s)\in\pi}(1-\delta_{\epsilon_{r}\epsilon_{s}})\lim_{n\rightarrow\infty}\frac{1}{n^{k+1}}\displaystyle\sum_{\mathit{I}_{2k}(\pi)}\displaystyle\prod_{r=1}^{2k}d_{i_{r}i_{r+1}}^{\epsilon_{r}}\delta_{i_r i_{\gamma \pi (r)}}.
\end{align*}
Let $\mathit{I}_{2k}'(\pi)=\{1\leq i_{1},i_{2}\dots,i_{2k} \leq n :  i_r= i_{\gamma \pi (r)} \text{ for }r=1,2,\dots,2k \}$. Then  
%\\~\\
%$ \lim_{n\rightarrow\infty}\varphi_{n} 
%    ({\mathit{A}}_{n}^{\epsilon_{1}} {\mathit{A}}_{n}^{\epsilon_{2}} \cdots {\mathit{A}}_{n}^{\epsilon_{2k}})$
\begin{align}
   \lim_{n\rightarrow\infty}\varphi_{n} 
   ({\mathit{A}}_{n}^{\epsilon_{1}}  \cdots {\mathit{A}}_{n}^{\epsilon_{2k}}) &= \displaystyle\sum_{\mathit{\pi}\in\mathit{NC}_{2}(2k)}\displaystyle\prod_{(r,s)\in\pi}(1-\delta_{\epsilon_{r}\epsilon_{s}})\lim_{n\rightarrow\infty}\frac{1}{n^{k+1}}\displaystyle\sum_{\mathit{I}_{2k}'(\pi)}\displaystyle\prod_{(r,s)\in\pi}d_{i_{r}i_{s}}^{\epsilon_{r}}d_{i_{s}i_{r}}^{\epsilon_{s}}
    \nonumber \\
    &= \displaystyle\sum_{\mathit{\pi}\in\mathit{NC}_{2}(2k)}\displaystyle\prod_{(r,s)\in\pi}(1-\delta_{\epsilon_{r}\epsilon_{s}})\lim_{n\rightarrow\infty}\frac{1}{n^{k+1}}\displaystyle\sum_{\mathit{I}_{2k}'(\pi)}\displaystyle\prod_{(r,s)\in\pi}d_{i_{r}i_{s}}^{\epsilon_{r}}d_{i_{r}i_{s}}^{\epsilon_{r}}
    \nonumber \\
    &=\displaystyle\sum_{\mathit{\pi}\in\mathit{NC}_{2}(2k)}\displaystyle\prod_{(r,s)\in\pi}(1-\delta_{\epsilon_{r}\epsilon_{s}})\lim_{n\rightarrow\infty}\frac{1}{n^{k+1}}\displaystyle\sum_{\mathit{I}_{2k}'(\pi)}\displaystyle\prod_{(r,s)\in\pi}d_{i_{r}i_{s}}^{\epsilon_{r}}.\label{eq1}
\end{align}
For $\pi = \{(r_{1},s_{1}),(r_{2},s_{2})\dots (r_{k},s_{k})\} \in\mathit{NC}_{2}(2k)$, we want to calculate the limit
$$\lim_{n\rightarrow\infty}\frac{1}{n^{k+1}}\sum_{\mathit{I}_{2k}'(\pi)}d_{i_{r_{1}}i_{s_{1}}}^{\epsilon_{r_{1}}}d_{i_{r_{2}}i_{s_{2}}}^{\epsilon_{r_{2}}}\cdots d_{i_{r_{k}}i_{s_{k}}}^{\epsilon_{r_{k}}},$$
if it exist. Recall that there are $k+1$ independent indices in $I_{2k}^{'}(\pi)$. Hence from Lemma \ref{lem 2}, for any $\epsilon > 0$ there exist $N_{\epsilon}\in \mathbb{N}$ such that
\begin{align*}
\sum_{\mathit{I}_{2k}'(\pi)}d_{i_{r_{1}}i_{s_{1}}}^{\epsilon_{r_{1}}}d_{i_{r_{2}}i_{s_{2}}}^{\epsilon_{r_{2}}}\cdots d_{i_{r_{k}}i_{s_{k}}}^{\epsilon_{r_{k}}} \geq (n-k\epsilon n)^{k+1},\text{ for all } n \geq N_\epsilon.
\end{align*}
Therefore we have 
$$\frac{1}{n^{k+1}}\sum_{\mathit{I}_{2k}'(\pi)}d_{i_{r_{1}}i_{s_{1}}}^{\epsilon_{r_{1}}}d_{i_{r_{2}}i_{s_{2}}}^{\epsilon_{r_{2}}}\cdots d_{i_{r_{k}}i_{s_{k}}}^{\epsilon_{r_{k}}} \geq (1-k\epsilon)^{k+1}, \text{ for all } n \geq N_\epsilon.$$
where k is fixed and $\epsilon > 0$ was arbitrary, which implies 
$$\liminf_{n}\frac{1}{n^{k+1}}\sum_{\mathit{I}_{2k}'(\pi)}d_{i_{r_{1}}i_{s_{1}}}^{\epsilon_{r_{1}}}d_{i_{r_{2}}i_{s_{2}}}^{\epsilon_{r_{2}}}\cdots d_{i_{r_{k}}i_{s_{k}}}^{\epsilon_{r_{k}}} \geq 1.$$
But we already know that,
$$\frac{1}{n^{k+1}}\sum_{\mathit{I}_{2k}'(\pi)}d_{i_{r_{1}}i_{s_{1}}}^{\epsilon_{r_{1}}}d_{i_{r_{2}}i_{s_{2}}}^{\epsilon_{r_{2}}}\cdots d_{i_{r_{k}}i_{s_{k}}}^{\epsilon_{r_{k}}} \leq 1.$$
Which implies
\begin{equation}\label{eq2}
	\lim_{n\rightarrow\infty}\frac{1}{n^{k+1}}\sum_{\mathit{I}_{2k}'(\pi)}d_{i_{r_{1}}i_{s_{1}}}^{\epsilon_{r_{1}}}d_{i_{r_{2}}i_{s_{2}}}^{\epsilon_{r_{2}}}\cdots d_{i_{r_{k}}i_{s_{k}}}^{\epsilon_{r_{k}}} = 1.
\end{equation}
Therefore from \eqref{eq1} and \eqref{eq2} we have
$$ \lim_{n\rightarrow\infty}\varphi_{n} 
 ({\mathit{A}}_{n}^{\epsilon_{1}} {\mathit{A}}_{n}^{\epsilon_{2}} \cdots {\mathit{A}}_{n}^{\epsilon_{2k}}) = \displaystyle\sum_{\mathit{\pi}\in\mathit{NC}_{2}(2k)}\displaystyle\prod_{(r,s)\in\pi}(1-\delta_{\epsilon_{r}\epsilon_{s}}).$$
 Which are the even moments of a circular variable, see Fact \ref{fact 4}.  Thus we get $A_n\stackrel{\ast}{\to}c$, as $n\to \infty.$ Hence, we have the first part.
 
 \vspace{.2cm}
\noindent \textit{Proof of the second part:}
 Forward direction holds trivially using the first part. For the converse, let $\tau_1,\tau_2,\dots,\tau_t \in [m]$, $\epsilon_1,\epsilon_2,\dots,\epsilon_t \in \{1,\ast\}$ and $Y_n^{(l)}=(y_{ij}^{(l)})=D_n^{(l)}\odot Y_n^{(l)}$ for $l=1,2,\dots,m$. Then 
 \begin{align*}
 	\varphi_n(A_n^{(\tau_1)\epsilon_1}\cdots A_n^{(\tau_t)\epsilon_t})&=\frac{1}{n^{\frac{t}{2}+1}}\sum_{I_t}\mathrm{E}[y_{i_1 i_2}^{(\tau_1)\epsilon_1}y_{i_2 i_3}^{(\tau_2)\epsilon_2}\cdots x_{i_t i_1}^{(\tau_t)\epsilon_t}].
 \end{align*}
 From Assumption \ref{assu 1}, the moments of all order are uniformly bounded for each random matrix and the number of matrices is finite, so we have the uniform bound for the moments of each order and random matrix. Using this and the same arguments used in the first part, for the odd $t$ we have
 $$
 \lim_{n \to \infty}\varphi_n(A_n^{(\tau_1)\epsilon_1}\cdots A_n^{(\tau_t)\epsilon_t})=0.
 $$  
 Now let $t=2k$, then again using the same arguments
 \begin{align*}
 	\varphi_n(A_n^{(\tau_1)\epsilon_1}\cdots A_n^{(\tau_t)\epsilon_t})&=\frac{1}{n^{k+1}}\displaystyle\sum_{\mathit{\pi}\in\mathcal{P}_{2}(2k)}\displaystyle\sum_{\mathit{I}_{2k}(\pi)}\displaystyle\prod_{(r,s)\in\pi}\mathrm{E}[y_{i_r i_{r+1}}^{(\tau_r)\epsilon_r}y_{i_s i_{s+1}}^{(\tau_s)\epsilon_s}] + o(1).
 \end{align*}
 We know that $X_n^{(1)},\ldots,X_n^{(m)}$ are independent and each satisfy Assumption \ref{assu 1}. Therefore we have
 \begin{align*}
 	\mathrm{E}[y_{i_r i_{r+1}}^{(\tau_r)\epsilon_r}y_{i_s i_{s+1}}^{(\tau_s)\epsilon_s}]&=d_{i_r i_{r+1}}^{(\tau_r)\epsilon_r}d_{i_s i_{s+1}}^{(\tau_s)\epsilon_s}\mathrm{E}[x_{i_r i_{r+1}}^{(\tau_r)\epsilon_r}x_{i_s i_{s+1}}^{(\tau_s)\epsilon_s}]
 	\\
 	&=d_{i_r i_{r+1}}^{(\tau_r)\epsilon_r}d_{i_s i_{s+1}}^{(\tau_s)\epsilon_s}(\delta_{\epsilon_{r} \epsilon_s}\delta_{i_r i_s}\delta_{i_{r+1} i_{s+1}} + (1-\delta_{\epsilon_{r} \epsilon_s})\delta_{i_r i_{s+1}}\delta_{i_s i_{r+1}} ) \delta_{\tau_r \tau_s}.
 \end{align*}
 Using Facts \ref{fact 1}, \ref{fact 2} and $|	\mathrm{E}[y_{i_r i_{r+1}}^{(\tau_r)\epsilon_r}y_{i_s i_{s+1}}^{(\tau_s)\epsilon_s}]|\leq 1$ we have
 \\
 $\lim_{n \to \infty }\varphi_n(A_n^{(\tau_1)\epsilon_1}\cdots A_n^{(\tau_{2k})\epsilon_{2k}})$
 \begin{align*}
 	&= \sum_{\pi \in NC_2(2k)} \lim_{n \to \infty}\frac{1}{n^{k+1}}\sum_{I_{2k}(\pi)} \prod_{(r,s) \in \pi} (1-\delta_{\epsilon_{r} \epsilon_s})\delta_{i_r i_{s+1}}\delta_{i_s i_{r+1}}d_{i_r i_{r+1}}^{(\tau_r)\epsilon_r}d_{i_s i_{s+1}}^{(\tau_s)\epsilon_s} \delta_{\tau_r \tau_s}
 	\\
 	&=\sum_{\pi \in NC_2(2k)}\prod_{(r,s) \in \pi} (1-\delta_{\epsilon_{r} \epsilon_s})\delta_{\tau_r \tau_s} \lim_{n \to \infty}\frac{1}{n^{k+1}}\sum_{I'_{2k}(\pi)}\prod_{(r,s) \in \pi} d_{i_r i_{r+1}}^{(\tau_r)\epsilon_r}d_{i_{s+1} i_{s}}^{(\tau_r)\epsilon_r}
 	\\
 	&= \sum_{\pi \in NC_2(2k)}\prod_{(r,s) \in \pi} (1-\delta_{\epsilon_{r} \epsilon_s})\delta_{\tau_r \tau_s} \lim_{n \to \infty}\frac{1}{n^{k+1}}\sum_{I'_{2k}(\pi)}\prod_{(r,s) \in \pi} d_{i_r i_{s}}^{(\tau_r)\epsilon_r},  	
 \end{align*}
 where $\mathit{I}_{2k}'(\pi)=\{1\leq i_{1},i_{2}\dots,i_{2k} \leq n :  i_r= i_{\gamma \pi (r)} \text{ for }r=1,2,\dots,2k \}$.  Then, as in the proof of the first part, from Lemma \ref{lem new} we get
 $$\liminf_n \frac{1}{n^{k+1}}\sum_{I'_{2k}(\pi)}\prod_{(r,s) \in \pi} d_{i_r i_{s}}^{(\tau_r)\epsilon_r} \geq 1.$$
 Therefore we have
 \begin{equation}\label{1}
 	\lim_{n \to \infty }\varphi_n(A_n^{(\tau_1)\epsilon_1}\cdots A_n^{(\tau_{2k})\epsilon_{2k}})=\sum_{\pi \in NC_2(2k)}\prod_{(r,s) \in \pi} (1-\delta_{\epsilon_{r} \epsilon_s})\delta_{\tau_r \tau_s}.
 \end{equation}
  Now let $c_1,\ldots,c_m$ are free circular variables. Then from Fact \ref{fact 4} we have
 % \begin{align*}
  %	\varphi(c_{\tau_1}^{\epsilon_1}\cdots c_{\tau_{2k}}^{\epsilon_{2k}})&= \sum_{\pi \in \mathit{NC}(2k)}\prod_{V\in\pi}\kappa(V)[c_{\tau_1}^{\epsilon_1},\ldots ,c_{\tau_{2k}}^{\epsilon_{2k}}].
 % \end{align*}
 % \sloppy Using the freeness and definition of circular variable, we have $\kappa_2(c_i,c_i^{\ast})=\kappa_2(c_i^{\ast},c_i)=1$ and all other free cumulants are zero. Therefore we have 
  \begin{equation}\label{2}
  	\varphi(c_{\tau_1}^{\epsilon_1}\cdots c_{\tau_{2k}}^{\epsilon_{2k}})=\sum_{\pi \in NC_2(2k)}\prod_{(r,s) \in \pi} (1-\delta_{\epsilon_{r} \epsilon_s})\delta_{\tau_r \tau_s}.
  \end{equation} 
  Therefore, from \eqref{1} and \eqref{2} we have the result.
 \end{proof} 

\begin{proof}[Proof of Theorem \ref{thm2}]
    Forward direction and the converse in the case of odd moments follows as in the proof of Theorem \ref{thm1}. Note that 
   \begin{align*}
   |\mathrm{E}[y_{i_{1}i_{2}}^{\epsilon_{1}}\cdots y_{i_{2k}i_{1}}^{\epsilon_{2k}}] |&=|d_{i_{1}i_{2}}^{\epsilon_{1}}\cdots d_{i_{2k}i_{1}}^{\epsilon_{2k}}\mathrm{E}[x_{i_{1}i_{2}}^{\epsilon_{1}}\cdots x_{i_{2k}i_{1}}^{\epsilon_{2k}}]| 
   \\
   &\leq |\mathrm{E}[x_{i_{1}i_{2}}^{\epsilon_{1}}\cdots x_{i_{2k}i_{1}}^{\epsilon_{2k}}]|. 
   \end{align*} 
    This fact and the arguments as in the proof of Theorem 1 in \cite{adhikari2019brown} imply that
  \begin{align*}
  	\varphi_{n} 
 ({\mathit{A}}_{n}^{\epsilon_{1}}  \cdots 
  {\mathit{A}}_{n}^{\epsilon_{2k}})
  &= \frac{1}{n^{k+1}}\displaystyle\sum_{\mathit{I}_{2k}}\mathrm{E}[y_{i_{1}i_{2}}^{\epsilon_{1}}y_{i_{2}i_{3}}^{\epsilon_{2}}\cdots y_{i_{2k}i_{1}}^{\epsilon_{2k}}] 
  \\
  &= \frac{1}{n^{k+1}}\displaystyle\sum_{\mathit{\pi}\in\mathcal{P}_{2}(2k)}\displaystyle\sum_{\mathit{I}_{2k}(\pi)}\displaystyle\prod_{(r,s)\in\pi}\mathrm{E}[y_{i_{r}i_{r+1}}^{\epsilon_{r}}y_{i_{s}i_{s+1}}^{\epsilon_{s}}] + o(1).
\end{align*}
Using the properties from Assumption \ref{assu 1} we have
\\
$\mathrm{E}[y_{i_{r}i_{r+1}}^{\epsilon_{r}}y_{i_{s}i_{s+1}}^{\epsilon_{s}}]$
\begin{align*}
         &=  d_{i_{r}i_{r+1}}^{\epsilon_{r}}d_{i_{s}i_{s+1}}^{\epsilon_{s}}\mathrm{E}[x_{i_{r}i_{r+1}}^{\epsilon_{r}}x_{i_{s}i_{s+1}}^{\epsilon_{s}}]
        \\
        &= d_{i_{r}i_{r+1}}^{\epsilon_{r}}d_{i_{s}i_{s+1}}^{\epsilon_{s}}\{\delta_{\epsilon_{r}\epsilon_{s}}(a'(r,s)+\rho b'(r,s))+(1-\delta_{\epsilon_{r}\epsilon_{s}})(b'(r,s)+\rho a'(r,s))\}
        \\
        &= d_{i_{r}i_{r+1}}^{\epsilon_{r}}d_{i_{s}i_{s+1}}^{\epsilon_{s}}\{[(1-\delta_{\epsilon_{r}\epsilon_{s}})+\rho\delta_{\epsilon_{r}\epsilon_{s}}]b'(r,s)+(\delta_{\epsilon_{r}\epsilon_{s}}+\rho(1-\delta_{\epsilon_{r}\epsilon_{s}}))a'(r,s)\}.
   \end{align*}
Note that $d_{i_{r}i_{r+1}}^{\epsilon_{r}}d_{i_{s}i_{s+1}}^{\epsilon_{s}}(\delta_{\epsilon_{r}\epsilon_{s}}+\rho(1-\delta_{\epsilon_{r}\epsilon_{s}})) \leq 1$. Therefore, using the Fact \ref{fact 2}, 
\\
$\displaystyle\lim_{n\rightarrow\infty}\varphi_{n} 
({\mathit{A}}_{n}^{\epsilon_{1}}  \cdots {\mathit{A}}_{n}^{\epsilon_{2k}})$
\begin{align*}
     &=\displaystyle\sum_{\mathit{\pi}\in\mathcal{P}_{2}(2k)}\lim_{n\rightarrow\infty}\frac{1}{n^{k+1}}\displaystyle\sum_{\mathit{I}_{2k}(\pi)}\displaystyle\prod_{(r,s)\in\pi}d_{i_{r}i_{r+1}}^{\epsilon_{r}}d_{i_{s}i_{s+1}}^{\epsilon_{s}}[(1-\delta_{\epsilon_{r}\epsilon_{s}})+\rho\delta_{\epsilon_{r}\epsilon_{s}}]b'(r,s)
    \\
    &= \displaystyle\sum_{\mathit{\pi}\in\mathcal{P}_{2}(2k)}\displaystyle\prod_{(r,s)\in\pi}[(1-\delta_{\epsilon_{r}\epsilon_{s}})+\rho\delta_{\epsilon_{r}\epsilon_{s}}]\lim_{n\rightarrow\infty}\frac{1}{n^{k+1}}\displaystyle\sum_{\mathit{I}_{2k}(\pi)}\displaystyle\prod_{r=1}^{2k}d_{i_{r}i_{r+1}}^{\epsilon_{r}}\delta_{i_{r}i_{\gamma \pi (r)}}.
\end{align*}
Using Fact \ref{fact 1} and $d_{i_{r}i_{r+1}}^{\epsilon_{r}}\delta_{i_{r}i_{\gamma \pi (r)}} \leq 1$ we have  
%\\~\\
%$ \lim_{n\rightarrow\infty}\varphi_{n} 
%    ({\mathit{A}}_{n}^{\epsilon_{1}}  \cdots {\mathit{A}}_{n}^{\epsilon_{2k}})$
\\
$\displaystyle\lim_{n\rightarrow\infty}\varphi_{n} 
({\mathit{A}}_{n}^{\epsilon_{1}}  \cdots {\mathit{A}}_{n}^{\epsilon_{2k}})$
\begin{align*}
      &= \displaystyle\sum_{\mathit{\pi}\in\mathit{NC}_{2}(2k)}\displaystyle\prod_{(r,s)\in\pi}[(1-\delta_{\epsilon_{r}\epsilon_{s}})+\rho\delta_{\epsilon_{r}\epsilon_{s}}]\lim_{n\rightarrow\infty}\frac{1}{n^{k+1}}\displaystyle\sum_{\mathit{I}_{2k}(\pi)}\displaystyle\prod_{r=1}^{2k}d_{i_{r}i_{r+1}}^{\epsilon_{r}}\delta_{i_{r}i_{\gamma \pi (r)}}
     \\
     &=\displaystyle\sum_{\mathit{\pi}\in\mathit{NC}_{2}(2k)}\displaystyle\prod_{(r,s)\in\pi}[(1-\delta_{\epsilon_{r}\epsilon_{s}})+\rho\delta_{\epsilon_{r}\epsilon_{s}}]\lim_{n\rightarrow\infty}\frac{1}{n^{k+1}}\displaystyle\sum_{\mathit{I}_{2k}'(\pi)}\displaystyle\prod_{r=1}^{2k}d_{i_{r}i_{r+1}}^{\epsilon_{r}},
\end{align*}
where, $\mathit{I}_{2k}'(\pi)=\{ i_{1},i_{2}\dots,i_{2k} \in I_{2k}(\pi) :  i_{r}=i_{\gamma \pi (r)} \text{ for } r=1,2,\dots,k \}$. Now $D_n$ is a symmetric matrix so we can remove the dependency of $d_{ij}$ from $\epsilon$ and interchange the indices. Therefore, we have 
%\\~\\
%$ \lim_{n\rightarrow\infty}\varphi_{n} 
%    ({\mathit{A}}_{n}^{\epsilon_{1}} {\mathit{A}}_{n}^{\epsilon_{2}} \cdots {\mathit{A}}_{n}^{\epsilon_{2k}})$
\\
$\displaystyle\lim_{n\rightarrow\infty}\varphi_{n} 
({\mathit{A}}_{n}^{\epsilon_{1}}  \cdots {\mathit{A}}_{n}^{\epsilon_{2k}})$
    \begin{align*}
          &=\displaystyle\sum_{\mathit{\pi}\in\mathit{NC}_{2}(2k)}\displaystyle\prod_{(r,s)\in\pi}[(1-\delta_{\epsilon_{r}\epsilon_{s}})+\rho\delta_{\epsilon_{r}\epsilon_{s}}]\lim_{n\rightarrow\infty}\frac{1}{n^{k+1}}\displaystyle\sum_{\mathit{I}_{2k}'(\pi)}\displaystyle\prod_{(r,s)\in\pi}d_{i_{r}i_{r+1}}d_{i_{s}i_{s+1}}
         \\
         &=\displaystyle\sum_{\mathit{\pi}\in\mathit{NC}_{2}(2k)}\displaystyle\prod_{(r,s)\in\pi}[(1-\delta_{\epsilon_{r}\epsilon_{s}})+\rho\delta_{\epsilon_{r}\epsilon_{s}}]\lim_{n\rightarrow\infty}\frac{1}{n^{k+1}}\displaystyle\sum_{\mathit{I}_{2k}'(\pi)}\displaystyle\prod_{(r,s)\in\pi}d_{i_{r}i_{s}}d_{i_{s+1}i_{s}}
         \\
         &=\displaystyle\sum_{\mathit{\pi}\in\mathit{NC}_{2}(2k)}\displaystyle\prod_{(r,s)\in\pi}[(1-\delta_{\epsilon_{r}\epsilon_{s}})+\rho\delta_{\epsilon_{r}\epsilon_{s}}]\lim_{n\rightarrow\infty}\frac{1}{n^{k+1}}\displaystyle\sum_{\mathit{I}_{2k}'(\pi)}\displaystyle\prod_{(r,s)\in\pi}d_{i_{r}i_{s}}.
\end{align*}
Now from the proof of Theorem \ref{thm1} we have
$$\lim_{n\rightarrow\infty}\frac{1}{n^{k+1}}\displaystyle\sum_{\mathit{I}_{2k}'(\pi)}\displaystyle\prod_{(r,s)\in\pi}d_{i_{r}i_{s}}=1.$$
Therefore, we have
\\
$$ \lim_{n\rightarrow\infty}\varphi_{n} 
    ({\mathit{A}}_{n}^{\epsilon_{1}}  \cdots {\mathit{A}}_{n}^{\epsilon_{2k}})=\displaystyle\sum_{\mathit{\pi}\in\mathit{NC}_{2}(2k)}\prod_{(r,s)\in\pi}[(1-\delta_{\epsilon_{r}\epsilon_{s}})+\rho\delta_{\epsilon_{r}\epsilon_{s}}]= \varphi(e^{\epsilon_{1}}\cdots e^{\epsilon_{2k}}),$$
 Hence using the Fact \ref{fact 4}, we have the result. The second part can be established as in Theorem \ref{thm1}, using the arguments given above. We skip the details.
\end{proof}

\begin{proof}[Proof of Lemma \ref{lem 1}] Suppose that the result is not true, then there exist $\epsilon_{o} > 0$ and a subsequence $n_{k}$ of natural numbers such that $|A_{\epsilon_{o},n_{k}}|<n_{k}-\epsilon_{o}n_{k}$,$\text{ for all } k \geq 1$. Then we have
 \begin{align*}
     \displaystyle\sum_{i,j=1}^{n_{k}}d_{ij} &= \displaystyle\sum_{i\in A_{\epsilon_{o},n_{k}}}\displaystyle\sum_{j=1}^{n}d_{ij} + \displaystyle\sum_{i\in A_{\epsilon_{o},n_{k}}^{c}}\displaystyle\sum_{j=1}^{n}d_{ij}
     \\
     &< n_{k}|A_{\epsilon_{o},n_{k}}| + (n_{k}-\epsilon_{o}n_{k})|A_{\epsilon_{o},n_{k}}^{c}|
     \\
     &= n_{k}|A_{\epsilon_{o},n_{k}}| + n_{k}|A_{\epsilon_{o},n_{k}}^{c}| - \epsilon_{o}n_{k}|A_{\epsilon_{o},n_{k}}^{c}|
     \\
     &= n_{k}^2 - \epsilon_{o}n_{k}(n_{k}-|A_{\epsilon_{o},n_{k}}|)
     \\
     &< n_{k}^2 - \epsilon_{o}n_{k}^2 + \epsilon_{o}n_{k}(n_{k}-\epsilon_{o}n_{k})
     \\
     &= n_{k}^2(1-\epsilon_{o}^2).
 \end{align*}
 This  implies that $\frac{1}{n_{k}^2} \sum_{i,j=1}^{n_{k}}d_{ij} < 1-\epsilon_{o}^2$, for all $k\geq 1$, which is a contradiction to the fact  $\frac{1}{n^{2}}\sum_{i,j=1}^{n}d_{ij}$ converges to $1$, as $n\to \infty$. Hence the result.
 \end{proof}

\begin{proof}[Proof of Lemma \ref{lem 2}]
        From Lemma \ref{lem 1} for any $\epsilon > 0$ there exist $N_\epsilon \in \mathbb{N}$ such that 
 \[
 |A_{\epsilon,n}|\geq n-\epsilon n \mbox{ and $|B_{\epsilon,n}|\geq n-\epsilon n$,$\text{ for all } n \geq N_\epsilon$.}
 \]
 Now consider, $A_{\epsilon,n}^{(1)}=\{i\in A_{\epsilon,n}: \sum_{j\in B_{\epsilon,n}}d_{ij}\geq n-2\epsilon n \}$ . Note that, for $ i \in A_{\epsilon,n}$ and $n \geq N_\epsilon$, we have
\begin{align*}
     \displaystyle\sum_{j \in B_{\epsilon,n}}d_{ij} &= \displaystyle\sum_{j=1}^{n}d_{ij}-\displaystyle\sum_{j \in B_{\epsilon,n}^{c}}d_{ij} 
      \geq n -\epsilon n - \epsilon n
     = n- 2 \epsilon n.
\end{align*}
Which implies $A_{\epsilon,n} \subseteq A_{\epsilon,n}^{(1)}$. Therefore
$$|A_{\epsilon,n}^{(1)}| \geq n-\epsilon n , \text{ for all } n \geq N_\epsilon.$$
Similarly, for $B_{\epsilon,n}^{(1)}=\{j\in B_{\epsilon,n}: \sum_{i\in A_{\epsilon,n}}d_{ij}\geq n-2\epsilon n \}$, by the same calculation we get $|B_{\epsilon,n}^{(1)}| \geq n- \epsilon n$,$\text{ for all } n \geq N_\epsilon$. For $l=2,3,\dots,k$ define
$$A_{\epsilon,n}^{(l)}=\{i\in A_{\epsilon,n}^{(l-1)}: \sum_{j\in B_{\epsilon,n}^{(l-1)}}d_{ij}\geq n-(l+1)\epsilon n \},$$  $$B_{\epsilon,n}^{(l)}=\{j\in B_{\epsilon,n}^{(l-1)}: \sum_{i\in A_{\epsilon,n}^{(l-1)}}d_{ij}\geq n-(l+1)\epsilon n \}.$$
Using the same argument as above, we have $|A_{\epsilon,n}^{(l)}| \geq n-\epsilon n$ and $|B_{\epsilon,n}^{(l)}| \geq n-\epsilon n$,$\text{ for all } n \geq N_\epsilon$. Then we want to calculate the following sum,
$$\displaystyle\sum_{I_{2k}}d_{i_{1}i_{2}}d_{i_{3}i_{4}}\cdots d_{i_{2k-1}i_{2k}}. $$
Observe that in the product there are $k$ many terms of the form $d_{i_{l}i_{l+1}}$ and the number of independent indices is $k+1$. Therefore, there exists at least one independent index that appears only in one of the $d_{i_{l}i_{l+1}}$ (after reducing to independent indices). Without loss of generality, let us assume that term is $d_{i_{2k-1}i_{2k}}$ and $i_{2k}$ is that index. Now we can break the sum by,
 
 \begin{align*}
   \displaystyle\sum_{I_{2k}}d_{i_{1}i_{2}}d_{i_{3}i_{4}}\cdots d_{i_{2k-1}i_{2k}} &= \displaystyle\sum_{i_{1},\dots , i_{2k-1} \in I_{2k}}d_{i_{1}i_{2}}d_{i_{3}i_{4}}\cdots d_{i_{2k-3}i_{2k-2}}\displaystyle\sum_{i_{2k}=1}^{n}d_{i_{2k-1}i_{2k}}
   \\
   & \geq (n-\epsilon n)\displaystyle\sum_{I_{2k}^{(1)}}d_{i_{1}i_{2}}d_{i_{3}i_{4}}\cdots d_{i_{2k-3}i_{2k-2}},
 \end{align*}
 where $I_{2k}^{(1)}=\{i_{1},i_{2},\dots,i_{2k-1} \in I_{2k} : i_{2k-1} \in A_{\epsilon,n} \}$. If $i_{2k-1}$ was that independent index, then take the sum over $i_{2k-1}$ and put the restriction that $i_{2k} \in B_{\epsilon,n}$. After this step, we are left with \(k\) independent indices in $I_{2k}^{(1)}$ and the number of terms in the product is \(k-1\). So again there will be one independent index which is only in one of the $d_{i_{l}i_{l+1}}$. Then restrict the corresponding index on $A_{\epsilon,n}^{(1)}$ or $B_{\epsilon,n}^{(1)}$ according to the position of the index. Let $I_{2k}^{(2)}$ be that restricted index set then we get the inequality,
$$\displaystyle\sum_{I_{2k}}d_{i_{1}i_{2}}d_{i_{3}i_{4}}\cdots d_{i_{2k-1}i_{2k}} \geq (n-\epsilon n)(n-2\epsilon n)\displaystyle\sum_{I_{2k}^{(2)}}d_{i_{l_{1}}i_{l_{2}}}d_{i_{l_{3}}i_{l_{4}}}\cdots d_{i_{l_{2k-5}}i_{l_{2k-4}}}.$$
Now, after repeating this process $k$ times and using the fact that $|A_{\epsilon,n}^{(l)}| \geq n-\epsilon n$ and $|B_{\epsilon,n}^{(l)}| \geq n-\epsilon n$,$\text{ for all } n \geq N_\epsilon$ and $l=2,3,\dots,k$, we get the following inequality, 
\begin{align*}
\displaystyle\sum_{I_{2k}}d_{i_{1}i_{2}}d_{i_{3}i_{4}}\cdots d_{i_{2k-1}i_{2k}} &\geq (n-\epsilon n)(n-2\epsilon n)\cdots (n-k\epsilon n) (n-\epsilon n)
\\
& \geq (n-k\epsilon n)^{k+1}.
\end{align*}
Hence the result.
\end{proof}

\begin{proof}[Proof of Lemma \ref{lem new}]
	 For $\epsilon>0$, apply Lemma \ref{lem 1} for each $D_{n}^{(1)},\dots,D_n^{(m)}$ and take the maximum of all stages. After that choose a one term from the product as done in the proof of Lemma \ref{lem 2} and restrict the corresponding index. Repeat the process $k$ times. Note that, in this case, the number of terms may be reduced by $\epsilon n$ at each step. However, we obtain a similar result with slightly different lower bound. We skip the details to avoid notational complexity.
	 %Then for $r=1,2,\dots,m$ and $l=1,2,\dots,k$ define,
%	\begin{align*}
%		A_{\epsilon,n}^{(\tau_r)(0)}&=\{1\leq i \leq n : \frac{1}{n}\sum_{j=1}^{n}d_{ij} \geq 1-\epsilon \}, 
%		\\
%		B_{\epsilon,n}^{(\tau_r)(0)}&=\{1\leq j \leq n : \frac{1}{n}\sum_{i=1}^{n}d_{ij} \geq 1-\epsilon \},
%		\\
%		A_{\epsilon,n}^{(\tau_r)(l)}&=\{i\in A_{\epsilon,n}^{(\tau_r)(l-1)}: \frac{1}{n}\sum_{j\in B_{\epsilon,n}^{(\tau_r)(l-1)}}d_{ij}\geq 1-(l+1)\epsilon  \},
%		\\
%		B_{\epsilon,n}^{(\tau_r)(l)}&=\{j\in B_{\epsilon,n}^{(\tau_r)(l-1)}: \frac{1}{n}\sum_{i\in A_{\epsilon,n}^{(\tau_r)(l-1)}}d_{ij}\geq 1-(l+1)\epsilon  \}.
%	\end{align*}
%	Then using above and the same arguments as in the proof of Lemma \ref{lem 2}, we conclude the result. 
\end{proof}

\section{Proof of Theorem \ref{thm3}}\label{sec5}
To prove the next result, we first state two lemmas and a key fact that will be used in the proof. We provide the proofs of the lemmas at the end of this section.

\begin{lemma}\label{lem 3}
    Let $D_{p \times n}=(d_{ij})_{p \times n}$ be a sequence of deterministic matrices whose entries satisfy Assumption \ref{assu 2} and  $\frac{p}{n}\xrightarrow{}y>0$ as $p\xrightarrow{}\infty$. Define 
$$ A_{\epsilon,n,p}=\{1\leq i \leq p : \frac{1}{n}\sum_{j=1}^n d_{ij}\geq 1-\epsilon \} \text{ and }B_{\epsilon,n,p}=\{1\leq j \leq n : \frac{1}{p}\sum_{i=1}^p d_{ij}\geq 1-\epsilon  \}.$$
Then for any $\epsilon>0$ there exists $N_\epsilon \in \mathbb{N}$ such that 
\[
|A_{\epsilon,n,p}|\geq p -\epsilon p \mbox{ and } |B_{\epsilon,n,p}|\geq n -\epsilon n, \text{ for all } n,p \geq N_{\epsilon}.
\]
\end{lemma}

\begin{lemma}\label{lem 4}
     Let $D_{p \times n}=(d_{ij})_{p \times n}$ be a sequence of deterministic matrices whose entries satisfy Assumption \ref{assu 2} and  $\frac{p}{n}\xrightarrow{}y>0$ as $p\xrightarrow{}\infty$. Suppose
      $I_{2k}=\{i_{r_1},i_{r_2},\dots \\,i_{r_{k}}, i_{s_1},i_{s_2},\dots,i_{s_k}: 1\leq i_{r_t} \leq p, 1 \leq i_{s_{t}} \leq n \text{ for } t=1,2,\dots,k \}$ be a set of $2k$ indices with exactly $k+1$ many independent indices. Then for any $\epsilon >0$ there exists $N_{\epsilon}\in \mathbb{N}$ such that
$$\sum_{I_{2k}}d_{i_{r_1}i_{s_1}}d_{i_{r_2}i_{s_2}}\cdots d_{i_{r_k}i_{s_k}} \geq p^u n^{k+1-u}(1-k\epsilon)^{k+1},\text{ for all } n,p \geq N_\epsilon,$$
where $u$ is the number of independent indices from $\{i_{r_1},i_{r_2},\dots,i_{r_{k}}\}$.
\end{lemma}

\begin{fact}(Theorem 5.5 of \cite{bose2010patterned})\label{fact 3}
    Suppose $Z_{p \times n}$ is a $p \times n$ rectangular matrix whose entries are i.i.d. $\mathit{N}(0,1)$ and 
    $\frac{p}{n} \to y > 0$ as $p \to \infty$. Then the ESD of $\overline{Z}_p = \frac{1}{n} Z_{p \times n} Z^*_{n \times p}$ 
    converges to the Marčenko–Pastur law.
\end{fact}

\begin{proof}[Proof of Theorem \ref{thm3}]
	{\it Necessary condition:}
    Suppose $\overline{X}_{p}\xrightarrow{\ast}\mathrm{MP}_y$. Then $\frac{1}{p}\mathrm{E}[\operatorname{Tr}(\overline{X}_p)]$ converges to the first moment of $\mathrm{MP}_y$ as $p\to \infty$. So first consider,

\begin{align*}
    \frac{1}{p}\mathrm{E}[\mathrm{Tr}(\overline{X}_p)] &= \frac{1}{p}\frac{1}{n} 
    \mathrm{E}[\operatorname{Tr}((D_{p \times n} \odot X_{p \times n})\cdot (D_{p \times n} \odot X_{p\times n})^\ast)] 
    \\
    &= \frac{1}{np}\sum_{\substack{1 \leq i \leq p, \\ 1 \leq j \leq n}}\mathrm{E}[x_{ij}d_{ij}x_{ji}^\ast d_{ji}^\ast]
    \\
    &= \frac{1}{np}\sum_{\substack{1 \leq i \leq p, \\ 1 \leq j \leq n}}d_{ij}^2\mathrm{E}[x_{ij}^2]
    \\
    &= \frac{1}{np}\sum_{\substack{1 \leq i \leq p, \\ 1 \leq j \leq n}}d_{ij}.
\end{align*}
The last equalty follows as $\mathrm{E}[x_ij^2]=1$ and $d_{ij}^2=d_{ij}$. On the other hand, the first moment of $\mathrm{MP}_y$ is,
 as $a=(1-\sqrt{y})^2$ and $b=(1+\sqrt{y})^2$,
\begin{align*}
    m_1 &= \frac{1}{2\pi y}\int_{a}^b \sqrt{(b-x)(x-a)}dx
    \\
    &= \frac{1}{2\pi y}\int_{-2\sqrt{y}}^{2\sqrt{y}} \sqrt{4y-z^2}dz \quad (x=1+y+z)
    \\
    &= 1.
\end{align*}
Therefore we have the forward direction, as $\frac{1}{p}\mathrm{E}[\mathrm{Tr}(\overline{X}_p)]\to m_1$. 

\vspace{.2cm}
\noindent {\it Sufficient condition:} Suppose $\frac{1}{np}\sum_{i,j}d_{ij}\to 1$ as $p\to \infty$.
By Fact \ref{fact 3} it is enough to prove the following for the converse part.
\[
\lim_{p \to \infty} \frac{1}{p} \mathrm{E}[\operatorname{Tr}(\overline{X}_p^k)] = 
\lim_{p \to \infty} \frac{1}{p} \mathrm{E}[\operatorname{Tr}(\overline{Z}_p^k)],\text{ for all } k\geq 1.
\]
Using the trace formula for product of matrices we have,
$$\frac{1}{p} \mathrm{E}[\operatorname{Tr}(\overline{X}_p^k)]=\frac{1}{p \cdot n^k} \sum_{I'_{2k}} \mathrm{E}[y^{\epsilon_1}_{i_1 i_2}
 y^{\epsilon_2}_{i_2 i_3} \cdots y^{\epsilon_{2k}}_{i_{2k} i_1}],$$
where  $I'_{2k} = \{(i_1, \ldots, i_{2k}) : 1 \le i_{2m} \le n,\, 1 \le i_{2m-1} \le p,\; m = 1, \ldots, k\}$, $y_{ij}=x_{ij}d_{ij}$ for all $ i,j \in I'_{2k}$,  $\epsilon_{2m}=\ast$ and $\epsilon_{2m-1}=1$ for $ m = 1, \ldots, k.$ Now using the same arguments as in the proof of Theorem 1 in \cite{adhikari2019brown}, we have
\begin{align*}
    \frac{1}{p} \mathrm{E}[\operatorname{Tr}(\overline{X}_p^k)]
    &=\frac{1}{p \cdot n^k} \sum_{I'_{2k}}  \sum_{\pi \in \mathcal{P}_{2}(2k)}\prod_{(r,s)\in \pi}
    \mathrm{E}[y^{\epsilon_r}_{i_r i_{r+1}} y^{\epsilon_s}_{i_s i_{s+1}}] + o(1)
    \\
    &= \frac{1}{p \cdot n^k} \sum_{I'_{2k}}  \sum_{\pi \in \mathcal{P}_{2}(2k)}\prod_{(r,s)\in \pi}\mathrm{E}[x^{\epsilon_r}_{i_r i_{r+1}} d^{\epsilon_r}_{i_r i_{r+1}} x^{\epsilon_s}_{i_s i_{s+1}}d^{\epsilon_s}_{i_s i_{s+1}}] + o(1)
    \\
    &= \frac{1}{p \cdot n^k} \sum_{I'_{2k}}  \sum_{\pi \in \mathcal{P}_{2}(2k)}\prod_{(r,s)\in \pi}d^{\epsilon_r}_{i_r i_{r+1}}d^{\epsilon_s}_{i_s i_{s+1}}\mathrm{E}[x^{\epsilon_r}_{i_r i_{r+1}}  x^{\epsilon_s}_{i_s i_{s+1}}] + o(1).
\end{align*}
Now consider,
\[
\begin{aligned}
\mathrm{E}[
x^{\epsilon_r}_{i_r i_{r+1}}
x^{\epsilon_s}_{i_s i_{s+1}}
] &= ( \delta_{i_r i_s} \delta_{i_{r+1} i_{s+1}} + \rho \, \delta_{i_r i_{s+1}} \delta_{i_{r+1} i_s} \, \delta_{\{i_r, i_{r+1} \leq \min\{p,n\}\}} ) \delta_{\epsilon_r \epsilon_s} \\
&\quad + ( \delta_{i_r i_{s+1}} \delta_{i_{r+1} i_s} + \rho \, \delta_{i_r i_s} \delta_{i_{r+1} i_{s+1}} \, \delta_{\{i_r, i_{r+1} \leq \min\{p,n\}\}} )(1 - \delta_{\epsilon_r \epsilon_s}) \\
&= ( \delta_{\epsilon_r \epsilon_s} + \rho (1 - \delta_{\epsilon_r \epsilon_s}) \delta_{\{i_r, i_{r+1} \leq \min\{p,n\}\}} ) \delta_{i_r i_s} \delta_{i_{r+1} i_{s+1}} \\
&\quad + ( \rho \, \delta_{\epsilon_r \epsilon_s} \, \delta_{\{i_r, i_{r+1} \leq \min\{p,n\}\}} + (1 - \delta_{\epsilon_r \epsilon_s}) ) \delta_{i_r i_{s+1}} \delta_{i_{r+1} i_s} \\
&= f(r,s) + g(r,s), \quad \text{say}.
\end{aligned}
\]
Observe that $f(r,s)\leq \delta_{i_r i_s} \delta_{i_{r+1} i_{s+1}}$ and $g(r,s)\leq \delta_{i_r i_{s+1}} \delta_{i_{r+1} i_s}$. So from Fact \ref{fact 2}, we have
\begin{align*}
    \lim_{p\to \infty}\frac{1}{p} \mathrm{E}[\operatorname{Tr}(\overline{X}_p^k)] &=\lim_{p\to \infty}\frac{1}{p \cdot n^k} \sum_{I'_{2k}}  \sum_{\pi \in \mathcal{P}_{2}(2k)}\prod_{(r,s)\in \pi}d^{\epsilon_r}_{i_r i_{r+1}}d^{\epsilon_s}_{i_s i_{s+1}}g(r,s)
    \\ 
    &= \sum_{\pi \in \mathit{NC}_{2}(2k)}\lim_{p\to \infty}\frac{1}{p \cdot n^k} \sum_{I'_{2k}}  \prod_{(r,s)\in \pi}d^{\epsilon_r}_{i_r i_{r+1}}d^{\epsilon_s}_{i_s i_{s+1}}g(r,s).
\end{align*}
Now observe that if $\pi\in\mathit{NC}_{2}(2k)$ then one of r and s has to be odd and other has to be even. Therefore $\delta_{\epsilon_r \epsilon_s}=0$ and we get,
\begin{align}\label{eqn:step1}
    \lim_{p\to \infty}\frac{1}{p} \mathrm{E}[\operatorname{Tr}(\overline{X}_p^k)] 
    &=\sum_{\pi \in \mathit{NC}_{2}(2k)}\lim_{p\to \infty}\frac{1}{p \cdot n^k} \sum_{I'_{2k}}  
    \prod_{(r,s)\in \pi}d^{\epsilon_r}_{i_r i_{r+1}}d^{\epsilon_s}_{i_s i_{s+1}}\delta_{i_r i_{s+1}} \delta_{i_{r+1} i_s}\nonumber
     \\
    &= \sum_{\pi \in \mathit{NC}_{2}(2k)}\lim_{p\to \infty}\frac{1}{p \cdot n^k} \sum_{I'_{2k}}  
    \prod_{r=1}^{2k}d^{\epsilon_r}_{i_r i_{r+1}}\delta_{i_r i_{\gamma \pi (r)}}\nonumber
    \\
    &= \sum_{\pi \in \mathit{NC}_{2}(2k)}\lim_{p\to \infty}\frac{1}{p \cdot n^k} \sum_{I''_{2k}(\pi)}  \prod_{r=1}^{2k}d^{\epsilon_r}_{i_r i_{r+1}},
\end{align}
where $I''_{2k}(\pi)=\{(i_1,i_2,\dots,i_{2k})\in I'_{2k} : \delta_{i_r i_{\gamma \pi (r)}}=1 \text{ for each } r=1,2,\dots,2k \}$.
\\
Let $\pi=\{\{r_1 ,s_1\},\{r_2 , s_2\},\dots,\{r_k ,s_k\}\}\in\mathit{NC}_2 (2k)$ be fixed and WLOG $r_i$'s are odd and $s_j$'s are even for each $i,j=1,2,\dots,k.$ Therefore $\epsilon_{r_i}=1$ and $\epsilon_{s_j}=\ast$. Then consider,   
\begin{align*}
    \sum_{I''_{2k}(\pi)}  \prod_{r=1}^{2k}d^{\epsilon_r}_{i_r i_{r+1}}&=  \sum_{I''_{2k}(\pi)}  d^{\epsilon_{r_1}}_{i_{r_1} i_{r_1+1}}d^{\epsilon_{r_2}}_{i_{r_2} i_{r_2+1}}\cdots d^{\epsilon_{r_k}}_{i_{r_k} i_{r_k+1}}d^{\epsilon_{s_1}}_{i_{s_1} i_{s_1+1}}d^{\epsilon_{s_2}}_{i_{s_2} i_{s_2+1}}\cdots d^{\epsilon_{s_k}}_{i_{s_k} i_{s_k+1}}
    \\
    &= \sum_{I''_{2k}(\pi)}  d_{i_{r_1} i_{r_1+1}}d_{i_{r_2} i_{r_2+1}}\cdots d_{i_{k_1} i_{r_k+1}}d_{i_{s_1+1} i_{s_1}}d_{i_{s_2+1} i_{s_2}}\cdots d_{i_{s_k+1} i_{s_k}}
    \\
    &= \sum_{I''_{2k}(\pi)}  d_{i_{r_1} i_{s_1}}d_{i_{r_2} i_{s_2}}\cdots d_{i_{r_k} i_{s_k}}d_{i_{r_1} i_{s_1}}d_{i_{r_2} i_{s_2}}\cdots d_{i_{r_k} i_{s_k}}
    \\
    &= \sum_{I''_{2k}(\pi)}  d_{i_{r_1} i_{s_1}}d_{i_{r_2} i_{s_2}}\cdots d_{i_{r_k} i_{s_k}}.
\end{align*}
Now from Fact \ref{fact 1}, number of independent indices in $I''_{2k}(\pi)$ is $k+1$. Therefore from Lemma \ref{lem 3} and Lemma \ref{lem 4}, for any $\epsilon>0$ there exists $N_\epsilon \in \mathbb{N}$ such that,
\begin{align*}
     \sum_{I''_{2k}(\pi)}  d_{i_{r_1} i_{s_1}}d_{i_{r_2} i_{s_2}}\cdots d_{i_{r_k} i_{s_k}} &\geq p^u n^{k+1-u}(1-k\epsilon)^{k+1},\text{ for all } n,p \geq N_\epsilon,
\end{align*}
where $u$ is the number of independent indices from $\{i_{r_1},i_{r_2},\dots,i_{r_{k}}\}$. Therefore
\begin{align*}
      \frac{1}{p \cdot n^k} \sum_{I''_{2k}(\pi)}  \prod_{r=1}^{2k}d^{\epsilon_r}_{i_r i_{r+1}} &\geq \frac{1}{p \cdot n^k} p^u n^{k+1-u}(1-k\epsilon)^{k+1},\text{ for all } n,p \geq N_\epsilon
      \\
      &= \left(\frac{p}{n}\right)^{u-1}(1-k\epsilon)^{k+1}.
\end{align*}
Which is true for any $\epsilon>0$. Therefore we have,
\begin{align}\label{eqn:lower}
    \liminf \frac{1}{p \cdot n^k} \sum_{I''_{2k}(\pi)}  \prod_{r=1}^{2k}d^{\epsilon_r}_{i_r i_{r+1}} \geq y^{u-1}.
\end{align}
Also we have,
\begin{align*}
    \frac{1}{p \cdot n^k} \sum_{I''_{2k}(\pi)}  \prod_{r=1}^{2k}d^{\epsilon_r}_{i_r i_{r+1}} &\leq \frac{1}{p \cdot n^k} \sum_{I''_{2k}(\pi)} 1 = \frac{1}{p \cdot n^k} p^u n^{k+1-u} = \left(\frac{p}{n}\right)^{u-1}.
\end{align*}
Which implies
\begin{align}\label{eqn:upper}
    \limsup  \frac{1}{p \cdot n^k} \sum_{I''_{2k}(\pi)}  \prod_{r=1}^{2k}d^{\epsilon_r}_{i_r i_{r+1}} \leq y^{u-1}.
\end{align}
Therefore from \eqref{eqn:step1}, \eqref{eqn:lower} and \eqref{eqn:upper} we have, 
\begin{align}\label{eqn:X}
  \lim_{p\to \infty}\frac{1}{p} \mathrm{E}[\operatorname{Tr}(\overline{X}_p^k)]=\sum_{\pi \in \mathit{NC}_{2}(2k)}y^{u(\pi)-1}.  
\end{align}
On the other hand, suppose $Z_{p\times n}=(z_{ij})_{p\times n}$. Then we have 
\begin{align*}
     \frac{1}{p} \mathrm{E}[\operatorname{Tr}(\overline{Z}_p^k)] &=\frac{1}{p \cdot n^k} \sum_{I'_{2k}} 
     \mathrm{E}[z^{\epsilon_1}_{i_1 i_2} z^{\epsilon_2}_{i_2 i_3} \cdots z^{\epsilon_{2k}}_{i_{2k} i_1}]
     \\
     &=\frac{1}{p \cdot n^k} \sum_{I'_{2k}}  \sum_{\pi \in \mathcal{P}_{2}(2k)}\prod_{(r,s)\in \pi}
     \mathrm{E}[z^{\epsilon_r}_{i_r i_{r+1}} z^{\epsilon_s}_{i_s i_{s+1}}].
\end{align*}
Where the last equality follows from Wick's formula. Using the fact that $z_{ij}$'s are i.i.d. $N(0,1)$ we have 
$$\mathrm{E}[z^{\epsilon_r}_{i_r i_{r+1}} z^{\epsilon_s}_{i_s i_{s+1}}]
=\delta_{\epsilon_r \epsilon_s}\delta_{i_r i_s}\delta_{i_{r+1}i_{s+1}} + (1-\delta_{\epsilon_r \epsilon_s})
\delta_{i_r i_{s+1}}\delta_{i_{r+1}i_{s}}.$$
From Fact \ref{fact 2}, we have  
\begin{align}\label{eqn:Z}
    \lim_{p \to \infty} \frac{1}{p} \mathrm{E}[\operatorname{Tr}(\overline{Z}_p^k)] &=  \sum_{\pi \in \mathit{NC}_{2}(2k)}
    \lim_{p\to \infty}\frac{1}{p \cdot n^k} \sum_{I'_{2k}}  \prod_{r=1}^{2k}\delta_{i_r i_{\gamma \pi (r)}}\nonumber
    \\
    &= \sum_{\pi \in \mathit{NC}_{2}(2k)}\lim_{p\to \infty}\frac{1}{p \cdot n^k} \sum_{I''_{2k}(\pi)} 1\nonumber
    \\
    &= \sum_{\pi \in \mathit{NC}_{2}(2k)}y^{u(\pi)-1}.
\end{align}
\smallskip
Using \eqref{eqn:X} and \eqref{eqn:Z} we have the result of the first part.

\noindent\textit{Proof of the second part:}
 Forward direction holds trivially using the first part. For the converse, suppose $\overline{Z}_p^{(1)},\overline{Z}_p^{(2)},
 \dots,\overline{Z}_p^{(m)}$ are $m$ independent copies of $\overline{Z}_p$(as defined in Fact \ref{fact 3}) then 
 $\overline{Z}_p^{(1)},\overline{Z}_p^{(2)},\dots,\overline{Z}_p^{(m)}$ are asymptotically free (see \cite{capitaine2004asymptotic}). 
 Therefore it is enough to prove that for any choice of $k\geq 1$ and $\tau_1,\tau_2,\dots,\tau_k \in [m]$ we have
$$
\lim_{p \to \infty}\varphi_n(\overline{X}_p^{(\tau_1)}\overline{X}_p^{(\tau_2)}\cdots\overline{X}_p^{(\tau_k)})
=\lim_{p \to \infty}\varphi_n(\overline{Z}_p^{(\tau_1)}\overline{Z}_p^{(\tau_2)}\cdots\overline{Z}_p^{(\tau_k)}).
$$
Therefore consider
\begin{align*}
	\varphi_n(\overline{X}_p^{(\tau_1)}\overline{X}_p^{(\tau_2)}\cdots\overline{X}_p^{(\tau_k)})
    &= \frac{1}{p}\mathrm{E}[\operatorname{Tr}(\overline{X}_p^{(\tau_1)}\overline{X}_p^{(\tau_2)}\cdots\overline{X}_p^{(\tau_k)})]
	\\
	&=\frac{1}{p\cdot n^k}\sum_{I_{2k}'}\mathrm{E}[y_{i_1 i_2}^{(\tau_1') \epsilon_1}y_{i_2 i_3}^{(\tau_2') \epsilon_2}\cdots y_{i_{2k} i_1}^{(\tau_{2k}') \epsilon_{2k}}]
	\\
	&=\frac{1}{p \cdot n^k} \sum_{I'_{2k}}  \sum_{\pi \in \mathcal{P}_{2}(2k)}\prod_{(r,s)\in \pi}
    \mathrm{E}[y_{i_r i_{r+1}}^{(\tau_r') \epsilon_r}y_{i_s i_{s+1}}^{(\tau_s') \epsilon_s}] + o(1),
\end{align*} 
where $I'_{2k} = \{(i_1, \ldots, i_{2k}) : 1 \le i_{2t} \le n,\, 1 \le i_{2t-1} \le p,\; t = 1, \ldots, k\}$, 
$y_{ij}=x_{ij}d_{ij}$, $\tau_{2t-1}'=\tau_{2t}'=\tau_t$, $\epsilon_{2t-1}=1$ and $\epsilon_{2t}=\ast$ for $t=1,2,\dots,k$. 
Then using the same argument as in first part we get
\\
$\lim_{p\to \infty}\varphi_n(\overline{X}_p^{(\tau_1)}\overline{X}_p^{(\tau_2)}\cdots\overline{X}_p^{(\tau_m)})$
\begin{align*}
	&=\sum_{\pi \in \mathit{NC}_{2}(2k)}\lim_{p\to \infty}\frac{1}{p \cdot n^k} \sum_{I'_{2k}}  \prod_{(r,s)\in \pi}d^{(\tau_r')\epsilon_r}_{i_r i_{r+1}}d^{(\tau_s')\epsilon_s}_{i_s i_{s+1}}\delta_{i_r i_{s+1}} \delta_{i_{r+1} i_s}\delta_{\tau_r' \tau_s'}
	\\
	&=\sum_{\pi \in \mathit{NC}_{2}(2k)}\prod_{(r,s)\in \pi}\delta_{\tau_r' \tau_s'}\lim_{p\to \infty}\frac{1}{p \cdot n^k} \sum_{I'_{2k}}  \prod_{(r,s)\in \pi}d^{(\tau_r')\epsilon_r}_{i_r i_{r+1}}d^{(\tau_r')\epsilon_r}_{i_{s+1} i_{s}}\delta_{i_r i_{s+1}} \delta_{i_{r+1} i_s}
	\\
	&=\sum_{\pi \in \mathit{NC}_{2}(2k)}\prod_{(r,s)\in \pi}\delta_{\tau_r' \tau_s'}\lim_{p\to \infty}\frac{1}{p \cdot n^k} \sum_{I''_{2k}(\pi)}  \prod_{(r,s)\in \pi}d^{(\tau_r')\epsilon_r}_{i_r i_{s}},
\end{align*}
where  $I''_{2k}(\pi)=\{(i_1,i_2,\dots,i_{2k})\in I'_{2k} : \delta_{i_r i_{\gamma \pi (r)}}=1 \text{ for each } r=1,2,\dots,2k \}$. 

Let $\pi=\{\{r_1 ,s_1\},\{r_2 , s_2\},\dots,\{r_k ,s_k\}\}\in\mathit{NC}_2 (2k)$ be fixed and WLOG $r_i$'s are odd and $s_j$'s are even for each $i,j=1,2,\dots,k.$ Therefore $\epsilon_{r_i}=1$ and $\epsilon_{s_j}=\ast$. Then we have 
$$\sum_{I''_{2k}(\pi)}  \prod_{(r,s)\in \pi}d^{(\tau_r')\epsilon_r}_{i_r i_{s}}=\sum_{I''_{2k}(\pi)}d^{(\tau_{r_1}')}_{i_{r_1} i_{s_1}}d^{(\tau_{r_2}')}_{i_{r_2} i_{s_2}}\cdots d^{(\tau_{r_k}')}_{i_{r_k} i_{s_k}}.$$ 
In this setup, an analogous to Lemma \ref{lem 4} holds and we get
$$\lim_{p\to \infty}\frac{1}{p \cdot n^k} \sum_{I''_{2k}(\pi)}  \prod_{(r,s)\in \pi}d^{(\tau_r')\epsilon_r}_{i_r i_{s}}=y^{u(\pi)-1},$$
where $u(\pi)$ denotes the  number of independent indices from $\{i_{r_1},i_{r_2},\dots,i_{r_{k}}\}$ in $\pi$. Therefore we have 
$$\lim_{p\to \infty}\varphi_n(\overline{X}_p^{(\tau_1)}\overline{X}_p^{(\tau_2)}\cdots\overline{X}_p^{(\tau_m)})=\sum_{\pi \in \mathit{NC}_{2}(2k)}y^{u(\pi)-1}\prod_{(r,s)\in \pi}\delta_{\tau_r' \tau_s'}.$$
On the other hand we have 
\begin{align*}
	\varphi_n(\overline{Z}_p^{(\tau_1)}\overline{Z}_p^{(\tau_2)}\cdots\overline{Z}_p^{(\tau_m)})
    &=\frac{1}{p\cdot n^k}\sum_{I_{2k}'}\mathrm{E}[z_{i_1 i_2}^{(\tau_1') \epsilon_1}z_{i_2 i_3}^{(\tau_2') \epsilon_2}
    \cdots z_{i_{2k} i_1}^{(\tau_{2k}') \epsilon_{2k}}]
	\\
	&=\frac{1}{p \cdot n^k} \sum_{I'_{2k}}  \sum_{\pi \in \mathcal{P}_{2}(2k)}\prod_{(r,s)\in \pi}
    \mathrm{E}[z_{i_r i_{r+1}}^{(\tau_r') \epsilon_r}z_{i_s i_{s+1}}^{(\tau_s') \epsilon_s}]. 
\end{align*}
As the random variables are $N(0,1)$, we have
\begin{align*}
	\lim_{p \to \infty}	\varphi_n(\overline{Z}_p^{(\tau_1)}\overline{Z}_p^{(\tau_2)}\cdots\overline{Z}_p^{(\tau_m)})
    &=\sum_{\pi \in \mathit{NC}_{2}(2k)}\prod_{(r,s)\in \pi}\delta_{\tau_r' \tau_s'}\lim_{p\to \infty}\frac{1}{p \cdot n^k} \sum_{I'_{2k}}  \prod_{r=1}^{2k}\delta_{i_r i_{\gamma \pi(r)}}
	\\
	&=\sum_{\pi \in \mathit{NC}_{2}(2k)}y^{u(\pi)-1}\prod_{(r,s)\in \pi}\delta_{\tau_r' \tau_s'}.
\end{align*}
Hence the result.
\end{proof}

\begin{proof}[Proof of Lemma \ref{lem 3}]
 On contrary, suppose there exist an $\epsilon_0 >0$ and subsequences $n_k,p_k$ such that $|A_{\epsilon_0,n_k,p_k}|< p_k -\epsilon_0 p_k$. Then consider,
\begin{align*}
    \sum_{\substack{1 \leq i \leq n_k \\ 1\leq j \leq p_k}}d_{ij} &= \sum_{A_{\epsilon_o,n_k,p_k}}\sum_{j=1}^{n_k} d_{ij} + \sum_{A_{\epsilon_o,n_k,p_k}^c}\sum_{j=1}^{n_k} d_{ij}
    \\
    &< n_k |A_{\epsilon_0,n_k,p_k}| + (n_k - \epsilon_0 n_k) |A_{\epsilon_o,n_k,p_k}^c|
    \\
    &=  n_k |A_{\epsilon_0,n_k,p_k}| + (n_k - \epsilon_0 n_k)(p_k - |A_{\epsilon_0,n_k,p_k}|)
    \\
    &= n_k |A_{\epsilon_0,n_k,p_k}| + (1-\epsilon_0)n_k p_k - n_k|A_{\epsilon_0,n_k,p_k}| \allowbreak + \epsilon_0 n_k|A_{\epsilon_0,n_k,p_k}| \allowbreak  
    \\
    &< (1-\epsilon_0)n_k p_k + \epsilon_0 n_k (p_k -\epsilon_0 p_k)
    \\
    &= n_k p_k (1-\epsilon_0 + \epsilon_0 - {\epsilon_0}^2).
\end{align*}
Therefore, for $k\ge 1$, we have
$$ \frac{1}{n_k p_k}\sum_{\substack{1 \leq i \leq n_k \\ 1\leq j \leq p_k}}d_{ij} < 1 - {\epsilon_0}^2 < 1.$$
Which is a contradiction to our hypothesis, therefore our assumption is wrong. One can work in a similar manner for the second inequality. Hence the result.
\end{proof}

\begin{proof}[Proof of Lemma \ref{lem 4}]
 For $l=1,\dots,k$ define,
$$A_{\epsilon,n,p}^{(l)}=\{i \in A_{\epsilon,n,p}^{(l-1)}  : \sum_{j\in B_{\epsilon,n,p}^{(l-1)}} d_{ij}\geq n-(l + 1)\epsilon n\},$$ $$B_{\epsilon,n,p}^{(l)}=\{j \in B_{\epsilon,n,p}^{(l-1)}  : \sum_{i\in A_{\epsilon,n,p}^{(l-1)}} d_{ij}\geq p-(l + 1)\epsilon p\},$$
where $A_{\epsilon,n,p}^{(0)}=A_{\epsilon,n,p}$ and $B_{\epsilon,n,p}^{(0)}=B_{\epsilon,n,p}$ are as defined in Previous Lemma.
From Lemma \ref{lem 3}, for any $\epsilon > 0$ there exists $N_\epsilon\in\mathbb{N}$ such that $|A_{\epsilon,n,p}|\geq p -\epsilon p$ and  $|B_{\epsilon,n,p}|\geq n -\epsilon n$,$\text{ for all } n,p \geq N_{\epsilon}$. Now for any $i\in A_{\epsilon,n,p}^{(l-1)}$,
\begin{align*}
    \sum_{j\in B_{\epsilon,n,p}^{(l-1)}} d_{ij} &= \sum_{j=1}^n d_{ij} - \sum_{j\in B_{\epsilon,n,p}^{(l-1)\, c}} d_{ij}
    \\
    &\geq n - l\epsilon n - \epsilon n
    \\
    &= n-(l+1)\epsilon n.
\end{align*}
Therefore we have 
$$|A_{\epsilon,n,p}^{(l)}| \geq p - \epsilon p,\text{ for all } n,p \geq N_{\epsilon}.$$
The last inequality holds from the principle of mathematical induction. Similarly we have $|B_{\epsilon,n,p}^{(l)}|\geq n -\epsilon n$,$\text{ for all } n,p \geq N_{\epsilon}$. Now in the following sum,
$$\sum_{I_{2k}}d_{i_{r_1}i_{s_1}}d_{i_{r_2}i_{s_2}}\cdots d_{i_{r_k}i_{s_k}},$$
observe that in the product there are $k$ many terms of the form $d_{i_{r_t}i_{s_t}}$ and 
the number of independent indices is $k+1$. 
Therefore, there exists at least one independent index that appears only in 
one of the $d_{i_{r_t}i_{s_t}}$ (after reducing to independent indices). 
Without loss of generality, let us assume that term is $d_{i_{r_k}i_{s_k}}$. Then there are two cases,
\begin{enumerate}
    \item[(i)] That index is $i_{r_k}$. Then we can break the sum by,    
    \\
    $\displaystyle\sum_{I_{2k}}d_{i_{r_1}i_{s_1}}d_{i_{r_2}i_{s_2}}\cdots d_{i_{r_k}i_{s_k}}$
    \begin{align*}
   &= \sum_{i_{r_1},i_{r_2},\dots,i_{r_{k-1}},i_{s_1},i_{s_2},\dots,i_{s_{k}} \in I_{2k}}d_{i_{r_1}i_{s_1}}d_{i_{r_2}i_{s_2}}\cdots d_{i_{r_{k-1}}i_{s_{k-1}}} \sum_{i_{r_k}=1}^p d_{i_{r_k}i_{s_k}}
        \\
        &\geq (p - \epsilon p)\sum_{ I_{2k}^{(1)}}d_{i_{r_1}i_{s_1}}d_{i_{r_2}i_{s_2}}\cdots d_{i_{r_{k-1}}i_{s_{k-1}}},
    \end{align*}
    where $I_{2k}^{(1)}=\{i_{r_1},i_{r_2},\dots,i_{r_{k-1}},i_{s_1},i_{s_2},\dots,i_{s_{k}} \in I_{2k} : i_{s_k}\in B_{\epsilon,n,p}^{(0)} \}$.
    \\
    \item[(ii)] That index is $i_{s_k}$. Then we can break the sum by,    
    \\
    $\displaystyle\sum_{I_{2k}}d_{i_{r_1}i_{s_1}}d_{i_{r_2}i_{s_2}}\cdots d_{i_{r_k}i_{s_k}} $
    \begin{align*}
        &= \sum_{i_{r_1},i_{r_2},\dots,i_{r_{k}},i_{s_1},i_{s_2},\dots,i_{s_{k-1}} \in I_{2k}}d_{i_{r_1}i_{s_1}}d_{i_{r_2}i_{s_2}}\cdots d_{i_{r_{k-1}}i_{s_{k-1}}} \sum_{i_{s_k}=1}^n d_{i{r_k}i_{s_k}}
        \\
        &\geq (n - \epsilon n)\sum_{ I_{2k}^{(1)}}d_{i_{r_1}i_{s_1}}d_{i_{r_2}i_{s_2}}\cdots d_{i_{r_{k-1}}i_{s_{k-1}}},
    \end{align*}
    where $I_{2k}^{(1)}=\{i_{r_1},i_{r_2},\dots,i_{r_{k}},i_{s_1},i_{s_2},\dots,i_{s_{k-1}} \in I_{2k} : i_{r_k}\in A_{\epsilon,n,p}^{(0)} \}$.
\end{enumerate}
After any of the above cases we have $I_{2k}^{(1)}$, which contains $k$ many independent indices and $k-1$ many terms in the product.
 Then  choose one term from $d_{i_{r_t}i_{s_t}}$ as done previously and restrict the corresponding index on corresponding $A_{\epsilon,n,p}^{(w)}$ or $B_{\epsilon,n,p}^{(w)}$. So after repeating this process $k$ times we will left with one of the following,
$$\sum_{i \in A_{\epsilon,n,p}^{(t)}}1 \geq p-\epsilon p \quad \text{or} \sum_{j \in B_{\epsilon,n,p}^{(t)}}1 \geq n-\epsilon n \quad \text{for some }1\leq t \leq k-1.$$
Therefore in total we have,
$$
\sum_{I_{2k}}d_{i_{r_1}i_{s_1}}d_{i_{r_2}i_{s_2}}\cdots d_{i_{r_k}i_{s_k}} 
\geq p^u n^{k+1-u}(1-k\epsilon)^{k+1},\text{ for all } n,p \geq N_\epsilon,
$$
where $u$ is the number of independent indices from $\{i_{r_1},i_{r_2},\dots,i_{r_{k}}\}$. This completes the proof.
\end{proof}

\bibliographystyle{abbrv}
\bibliography{mybib}
\end{document}